\documentclass[11pt]{amsart}
\usepackage[utf8]{inputenc}
\usepackage{amsmath,amssymb}
\usepackage{enumerate}
\usepackage[pagebackref,colorlinks=true,linkcolor=blue,citecolor=blue,urlcolor=red, hypertexnames=true]{hyperref}
\makeatletter
\@namedef{subjclassname@2020}{\textup{2020} Mathematics Subject Classification}
\makeatother

\numberwithin{equation}{section}

\newtheorem{theorem}{Theorem}[section]

\newtheorem{lemma}[theorem]{Lemma}

\newtheorem{proposition}[theorem]{Proposition}

\newtheorem{corollary}[theorem]{Corollary}

\newtheorem{claim}[theorem]{Claim}

\newtheorem{fact}[theorem]{Fact}

{\theoremstyle{definition}}
{\theoremstyle{definition}}

{\theoremstyle{definition}\newtheorem{example}[theorem]{Example}}

\newtheorem{observation}[theorem]{Observation}

{\theoremstyle{definition}\newtheorem{definition}[theorem]{Definition}}

{\theoremstyle{definition}}

{\theoremstyle{definition}\newtheorem{remark}[theorem]{Remark}}

\newtheorem{question}[theorem]{Question}

{\theoremstyle{definition}
\newtheorem*{remark*}{\it Remark}
\newtheorem*{example*}{\it Example}}

\def\L{\mathcal L}

\def\glx{{\rm GL}(X)}

\def\effhc{{\mathcal F}\hbox{-}\mathrm{HC}}
\def\deffhc{d\,\hbox{-}\mathcal F\hbox{-\rm HC}}
\newcommand{\CC}{\mathbb{C}}
\newcommand{\NN}{\mathbb{N}}
\newcommand{\DD}{\mathbb{D}}
\newcommand{\RR}{\mathbb{R}}
\newcommand{\ZZ}{\mathbb{Z}}
\newcommand{\TT}{\mathbb{T}}

\newcommand{\veps}{\varepsilon}

\newcommand{\ldens}{\underline{\textrm{dens}}}
\newcommand{\udens}{\overline{\textrm{dens}}}

\newcommand{\sbt}{\,\begin{picture}(-1,1)(-1,-3)\circle*{3}\end{picture}\ }

\newcommand{\hc}{\mathrm{HC}}
\newcommand{\fhc}{\mathrm{FHC}}
\newcommand{\ufhc}{\hbox{UFHC}}
\newcommand{\dfhc}{d\hbox{\,-\,}\mathrm{FHC}}
\newcommand{\dhc}{d\hbox{\,-\,}\mathrm{HC}}
\newcommand{\wick}[1]{\pmb:{#1}\pmb:}
\newcommand{\tvw}{T_{v,\,w,\, \varphi, \,b\,}}
\newcommand{\gn}{n\ge 1}

\newcommand{\cpt}{C$_{+}$-type}
\newcommand{\cput}{C$_{+,1}$-type}

\def\ds{\displaystyle}

\begin{document}
\title[Hereditary frequent hypercyclicity]
{Hereditarily frequently hypercyclic operators \\ and disjoint frequent hypercyclicity}

\author[F. Bayart]{Fr\'ed\'eric Bayart}
\address[F. Bayart]{Laboratoire de Math\'ematiques Blaise Pascal UMR 6620 CNRS, Universit\'e Clermont Auvergne, Campus universitaire des C\'ezeaux, 3 place Vasarely, 63178 Aubi\`ere Cedex, France.}
\email{frederic.bayart@uca.fr}
\author[S. Grivaux]{Sophie Grivaux}
\address[S. Grivaux]{CNRS, Univ. Lille, UMR 8524 - Laboratoire Paul
Painlev\'e, F-59000 Lille, France}
\email{sophie.grivaux@univ-lille.fr}
\author[\'{E}. Matheron]{\'{E}tienne Matheron}
\address[\'{E}. Matheron]{Univ. Artois, UR 2462 - Laboratoire de Math\'{e}matiques de Lens (LML)\\ F-62300 Lens, France}
\email{etienne.matheron@univ-artois.fr}
\author[Q. Menet]{Quentin Menet}
\address[Q. Menet]{Service de Probabilit\'e et Statistique, D\'epartement de Math\'ematique\\ Universit\'{e} de Mons\\ Place du Parc 20\\ 7000 Mons, Belgium}
\email{quentin.menet@umons.ac.be}

\thanks{We gratefully acknowledge the support of the Mathematisches Forschungsinstitut Oberwolfach, where part of this work was carried out. The second author was also supported in part by the Labex CEMPI (ANR-11-LABX-0007-01). The fourth author is a Research Associate of the Fonds de la Recherche Scientifique - FNRS}



\begin{abstract}
We introduce and study the notion of hereditary frequent hypercyclicity, which is a reinforcement of the well known concept of frequent hypercyclicity. This notion is useful for the study of the dynamical properties of direct sums of operators; in particular, a basic observation is that the direct sum of a hereditarily frequently hypercyclic operator with any frequently hypercyclic operator is frequently hypercyclic. Among other results, we show that operators satisfying the Frequent Hypercyclicity Criterion are hereditarily frequently hypercyclic, as well as a large class of operators whose unimodular eigenvectors are spanning with respect to the Lebesgue measure.
On the other hand, we exhibit two frequently hypercyclic weighted shifts $B_w,B_{w'}$ on  $c_0(\ZZ_+)$ whose direct sum $B_w\oplus B_{w'}$ is not $\mathcal{U}$-frequently hypercyclic (so that neither of them is hereditarily frequently hypercyclic), and we construct a $C$-type operator on $\ell_p(\ZZ_+)$, $1\le p<\infty$  which is frequently hypercyclic but not hereditarily frequently hypercyclic. We also solve several problems concerning disjoint frequent hypercyclicity: we show that for every $N\in\NN$, any disjoint frequently hypercyclic $N$-tuple of operators $(T_1,\dots ,T_N)$ can be extended to a disjoint frequently hypercyclic $(N+1)$-tuple $(T_1,\dots ,T_N, T_{N+1})$ as soon as the underlying space supports a hereditarily frequently hypercyclic operator; we construct a disjoint frequently hypercyclic pair which is not densely disjoint hypercyclic; and we show that the pair $(D,\tau_a)$ is disjoint frequently hypercyclic, where $D$ is the derivation operator acting on the space of entire functions and $\tau_a$ is the operator of translation by $a\in\CC\setminus\{ 0\}$.
Part of our results are in fact obtained in the general setting of Furstenberg families.
\end{abstract}
\maketitle

\section{Introduction}

This paper is devoted to two different topics, both pertaining to the study of the dynamics of linear operators. Firstly, motivated by some questions regarding the behaviour of direct sums of operators, we introduce a new dynamical property of continuous linear operators on Banach or Fr\'echet spaces, which appears to be a very natural strengthening of the classical notion of frequent hypercyclicity; we call it \emph{hereditary frequent hypercyclicity}. We believe that this is an interesting notion, and  we study it in some detail. Secondly, we address some questions concerning \emph{disjoint frequent hypercyclicity} -- also called \emph{diagonal} frequent hypercyclicity. One notable connection between these two topics is that hereditarily frequently hypercyclic operators can be used to extend  diagonally frequently hypercyclic tuples (see below). 

\smallskip
In what follows, the letter $X$ denotes an infinite-dimensional Polish topological vector space, and  $\mathfrak L(X)$ is the space of  continuous linear operators on $X$. Recall that an operator $T\in\mathfrak L(X)$ is said to be hypercyclic if it has a dense orbit, \textit{i.e.} there exists $x\in X$ such that $\{T^n x:\ n\geq 0\}$ is dense in $X$;  equivalently, for each non-empty open set $V\subset X$, the ``visit set'' $\mathcal  N_T(x,V):=\{ n\in\NN;\; T^nx\in V\}$ is infinite. A much stronger property, introduced in \cite{BAYGRITAMS}, is \emph{frequent} hypercyclicity: the operator $T$ is frequently hypercyclic if there exists $x\in X$ such that for each non-empty open set $V\subset X$, the set $\mathcal N_T(x,V)$ has positive lower density. We refer the reader to \cite{BM09} and \cite{KarlAlfred} for an in-depth presentation of various aspects of linear dynamics.

\par\smallskip
More recently, quantitative notions of hypercyclicity have begun to be studied in a very general framework (\cite{BoGre18}, \cite{BeMePePu}, \cite{EEM}). Let $\mathcal F$ be a \textbf{Furstenberg family}, \textit{i.e.} a family of non-empty subsets of $\NN$ which is hereditary upwards (if $A'\supset A\in \mathcal F$ then $A'\in\mathcal F$). Following \cite{BeMePePu}, we say that an operator $T\in\mathfrak L(X)$ is \textbf{$\mathcal F$-hypercyclic} if there exists $x\in X$, called a $\mathcal F$-hypercyclic vector for $T$,  such that for each non-empty open set $V\subset X$, the set $\mathcal N_T(x,V)$ belongs to $\mathcal F$. Thus, hypercyclicity corresponds to the family of all infinite subsets of $\NN$, and frequent hypercyclicity corresponds to the family of sets with positive lower density. The set of $\mathcal F$-hypercyclic vectors for $T$ will be denoted by $\effhc(T)$. However, in accordance with a well-established notation, we write $\hc(T)$ in the hypercyclic case and $\fhc(T)$ in the frequently hypercyclic case. Also, when $\mathcal F$ is the family of all subsets of 
 $\NN$ with positive upper density, we say that $T$ is \emph{$\mathcal U$-frequently hypercyclic} and we write $\ufhc(T)$.

\medskip
The starting point of the paper is the following question: 

\begin{question}\label{q0}
Let $T_1\in\mathfrak L(X_1)$ and $T_2\in\mathfrak L(X_2)$ be two frequently hypercyclic operators; is it true that $T_1\oplus T_2$ is frequently hypercyclic?  
\end{question}

This question seems to have been considered for the first time in \cite[Section 8]{GMM21}, and appears as a natural variant of the following well-known open problem in linear dynamics \cite{BAYGRITAMS}: if $T$ is a frequently hypercyclic operator, is it true that $T\oplus T$ is frequently hypercyclic?
Question \ref{q0} makes sense for $\mathcal F$-hypercyclicity as well; and in the especially interesting case $T_1=T_2$, the answer is known for the family of all infinite subsets of $\NN$ and for the family of sets with positive \emph{upper} density. Indeed, a famous example from \cite{DlRR} shows that hypercyclicity of $T$ does not imply that of $T\oplus T$, whereas it is proved in \cite{EEM} that {$\mathcal U$-frequent hypercyclicity} of $T$ does imply that of $T\oplus T$. 

\smallskip
Given $T_1\in\mathfrak L(X_1)$ and $T_2\in\mathfrak L(X_2)$ two frequently hypercyclic operators, a natural way to show that $T_1\oplus T_2$ is frequently hypercyclic  would be the following. Let $(V_i)_{i\in\NN}$ be a countable basis of open sets for $X_1\times X_2$, and assume that each $V_i$ has the form $V_i=V_{i,1}\times V_{i,2}$, where $V_{i,1}$ is open in $X_1$ and $V_{i,2}$ is open in $X_2$. Pick a frequently hypercyclic vector $x_1\in X_1$ for $T_1$. Then, for any $i\in\NN$, there exists a set $A_i\subset \NN$  with positive lower density
such that $T_1^{n}x\in V_{i,1}$ for all $n\in A_i$. We would be done if we were able to find a vector $x_2\in X_2$ with the following property: for every $i\in\NN$, there exists a set $B_i$ with positive lower density \emph{and contained in  $A_i$} (this is the important point, which cannot be guaranteed if $x_2$ is simply assumed to be frequently hypercyclic for $T_2$)  such that $T_2^{n}x_2\in V_{i,2}$ for all $n\in B_i$. This leads to the following definition. 

\begin{definition}\label{def0}
 Let $\mathcal F\subset 2^{\NN}$ be a Furstenberg family. We say that an operator $T\in \mathfrak L(X)$ is  {\bf hereditarily $\mathcal F$-hypercyclic}
 if, for any countable family $(V_i)_{i\in I} $ of non-empty open subsets of $X$ and any family $(A_i)_{i\in I}\subset \mathcal F$ indexed by the same countable set $I$, 
 there exists a vector $x\in X$ such that  $\mathcal N_T(x,V_i)\cap A_i\in\mathcal F$ for every $i\in I$; in other words, for each $i\in I$, there is a set $B_i\in\mathcal F$ such that $B_i\subset A_i$ and $T^nx\in V_i$ for all $n\in B_i$.
 \end{definition}
 
 When $\mathcal F$ is the family of sets with positive lower density, we say (of course) that the operator $T$ is \emph{hereditarily frequently hypercyclic}; and likewise for $\mathcal U$-frequent hypercyclicity. By the above discussion, we get 
 \begin{observation}\label{prop:sumhfhc0}
 Let $\mathcal F\subset 2^\NN$ be a Furstenberg family. If $T_1, T_2$ are two $\mathcal F$-hypercyclic operators and at least one of them is hereditarily $\mathcal F$-hypercyclic, then $T_1\oplus T_2$ is $\mathcal F$-hypercyclic.
\end{observation}

Note that when $\mathcal F$ is the family of all infinite subsets of $\NN$,  hereditary $\mathcal F$-hypercyclicity is equivalent to \emph{topological mixing}; see Section \ref{weak} for the (easy) proof. So Observation \ref{prop:sumhfhc0} implies in particular that the direct sum of a hypercyclic operator with a topologically mixing operator is hypercyclic; this is of course well known.

\smallskip We also point out that -- perhaps surprisingly -- hereditary $\mathcal F$-hypercyclicity automatically implies \emph{dense} hereditary $\mathcal F$-hypercyclicity: given $(A_i)$ and $(V_i)$ as in Definition \ref{def0} above, there is a dense set of vectors $x\in X$ satisfying the required property; see Proposition \ref{densely} below.

\medskip

Having introduced a definition, we are immediately faced with some obvious questions.

\smallskip

 {\sbt\ \it Are there any hereditarily frequently hypercyclic operators?} We answer this in the affirmative, by two different methods. Indeed, there are two ``standard'' ways of proving that an operator is frequently hypercyclic: either by showing that it satisfies the so-called Frequent Hypercyclicity Criterion (see \cite{BM09} or  \cite{KarlAlfred}) or by exhibiting a large supply of eigenvectors associated to unimodular eigenvalues (see e.g. \cite{BAYGRITAMS} or  \cite{BAYMATHERGOBEST}). It turns out that in both cases, one gets in fact hereditary frequent hypercyclicity (Theorem \ref{thm:operatorfhc}, Theorem \ref{thm:ergodic}) or hereditary $\mathcal U$-frequent hypercyclicity (Theorem \ref{HUFHC}).

\smallskip

{\sbt\ \it Is hereditary frequent hypercyclicity a new notion?} In other words, are there any frequently hypercyclic operators which are not hereditarily frequently hypercyclic? The answer is ``Yes", and we prove this in two ways. On the one hand, we construct two frequently hypercyclic weighted shifts $B_w$ and $B_{w'}$ on $c_0(\mathbb Z_+)$ such that $B_w\oplus B_{w'}$ is not $\mathcal U$-frequently hypercyclic (Theorem \ref{thm:sumfhc}), so that neither of them can be hereditarily frequently hypercyclic by Observation \ref{prop:sumhfhc0}. This also gives a strong negative answer to the $T_1\oplus T_2$ frequent hypercyclicity problem of Question \ref{q0}. On the other hand, with the terminology of \cite{GMM21}, we construct a C-type  operator on $\ell_p(\ZZ_+)$, $1\leq p<\infty$ which is frequently hypercyclic but not hereditarly frequently hypercyclic (Theorem \ref{Ctypeex}).

\smallskip
 {\sbt\ \it  What are hereditarily frequently hypercyclic operators good for?} We will use them in the context of ``disjoint hypercyclicity''. 
 The notion of disjointness in  linear dynamics was introduced independently in \cite{Bernal-disjoint} and \cite{BePe-disjoint}. Let $N\geq 1$ and let $T_1,\dots,T_N\in\mathfrak L(X)$. Following \cite{BePe-disjoint}, we say that $T_1,\dots,T_N$ are \textbf{disjoint}, or that the tuple $(T_1,\dots ,T_N)$ is \textbf{diagonally hypercyclic}, if there exists $x\in X$ such that the set $\{(T_1^n x,\dots,T_N^n x):\ n\geq 0\}$ is dense in $X^N$; in other words, the ``diagonal'' vector $x\oplus\cdots\oplus x$ is hypercyclic for $T_1\oplus\cdots\oplus T_N$. Such a vector $x$ is said to be $d$-hypercyclic for the tuple $(T_1,\dots ,T_N)$, and the set of $d$-hypercyclic vectors for $(T_1,\dots,T_N)$ will be denoted by $\dhc(T_1,\dots,T_N)$.
 Similarly, $(T_1,\dots,T_N)$ is said to be \textbf{$d$-frequently hypercyclic} if there exists $x\in X$ such that $x\oplus\cdots\oplus x$ is a frequently hypercyclic
vector for $T_1\oplus\cdots\oplus T_N$, and we denote by $\dfhc(T_1,\dots,T_N)$ 
the set of $d$-frequently hypercyclic vectors for $(T_1,\dots ,T_N)$.

A natural problem regarding $d$-hypercyclicity is that of the \emph{extension} of $d$-hypercyclic tuples. It was shown in  \cite{MaSa24} that given any $N\geq 1,$ any Banach space $X$ and any $T_1,\dots,T_N\in\mathfrak L(X)$ such that $(T_1,\dots,T_N)$ is $d$-hypercyclic, there exists $T_{N+1}\in\mathfrak L(X)$ such that $(T_1,\dots,T_{N+1})$ is also $d$-hypercyclic. As for $d$-frequent hypercyclicity, the situation is trickier since there exist Banach spaces which do not support any frequently hypercyclic operator (\cite{Shk09}). The best one could hope for is that as soon as $X$ supports a frequently hypercyclic operator, then one can extend $d$-frequently hypercyclic tuples. We are unable to prove this, but we show that one can indeed extend $d$-frequently hypercyclic tuples as soon as $X$ supports a \emph{hereditarily} frequently hypercyclic operator (Theorem \ref{thm:extendingdfhc2}).

\smallskip
The preceding discussion has outlined the content of Sections \ref{FHCC}--\ref{extend} of the paper, and  hopefully it is clear that Section \ref{extend} makes a transition between our two topics --  hereditary frequent hypercyclicity and $d$-frequent hypercyclicity. The next two sections are exclusively devoted to $d$-frequent hypercyclicity.
In Section \ref{pasdensely}, we show (in the spirit of \cite{SaSh14}) that there exists a $d$-frequently hypercyclic pair $(T_1,T_2)$ on some Banach space $X$ which is not densely $d$-hypercyclic, \textit{i.e.} the set $\dhc(T_1,T_2)$ is not dense in $X$ (Theorem \ref{thm:dfhcnotddens}). In Section \ref{ajout}, we give a sufficient condition for $d$-frequent hypercyclicity of a tuple $(T_1,\dots ,T_N)$ in terms of eigenvectors of the operators $T_i$ (Theorem \ref{tambouille}); and this allows us for example to show that the pair $(D, \tau_a)$ is $d$-frequently hypercyclic, where $D$ is the derivation operator on the space of entire functions $H(\CC)$ and $\tau_a$ is the operator of translation by $a\in\CC\setminus\{ 0\}$.

\smallskip Finally, Section \ref{quest} contains a few additional remarks and a number of open questions originating in a rather natural way from our work.

\section{The Frequent Hypercyclicity Criterion}\label{FHCC}
The Frequent Hypercyclicity Criterion (FHCC) is a very efficient tool for showing that a given operator is frequently hypercyclic. Since \cite{BAYGRITAMS}, there have been several versions of it in the literature; we choose here the most widely used (\cite{BoGre07}): an operator $T\in\mathfrak L(X)$ satisfies the FHCC provided there exist a dense set $\mathcal D\subset X$ and a map $S:\mathcal D\to\mathcal D$ such that 
\begin{itemize}
\item[\sbt] $TS=I$ on $\mathcal D$;
\item[\sbt] for any $x\in \mathcal D,$ the series $\sum T^n x$ and $\sum S^n x$ are unconditionally convergent.
\end{itemize}

\smallskip In this section, we show that the FHCC implies in fact hereditary frequent hypercyclicity.
\begin{theorem}\label{thm:operatorfhc}
 If $T\in\mathfrak L(X)$satisfies the Frequent Hypercyclicity Criterion, 
then $T$ is hereditarily frequently hypercyclic.
\end{theorem}

\smallskip 
The proof of this theorem will closely mimic the classical proof that an operator satisfying the FHCC is frequently hypercyclic. Recall that the latter depends on the construction
of subsets of $\NN$ with positive lower density which are ``well separated''. 
To obtain hereditary  frequent hypercyclicity, we need to control more precisely
these subsets, and in particular we have to be sure that one can find them inside some
prescribed subsets of $\NN$ (of positive lower density, of course). This is the content
of the next lemma, which is useful in other situations as well (see \cite{MaMePu} and the very recent \cite{Rodrigo}). This lemma is actually contained in \cite[Lemma 2.2]{MaMePu}, but we give a proof for completeness (and convenience of the reader).

\begin{lemma}\label{prop:subsets}
 Let $(A_i)_{i\in I}$ be a countable family of subsets of $\NN$ with positive lower density, and let $(N_i)_{i\in I}$ be a family 
 of positive integers indexed by the same countable set $I$. There exists a family $(B_i)_{i\in I}$ of pairwise disjoint subsets of $\NN$ with positive lower density such that
 \begin{itemize}
  \item[\rm (a)] $B_i\subset A_i$ and $\min(B_i)\geq N_i$ for all $i\in I$;
  \item[\rm (b)] for any $i, j\in I$ and any $(n,m)\in B_i\times B_j$ with $n\neq m,$ $|n-m|\geq N_i+N_j.$
 \end{itemize}
\end{lemma}
\begin{proof}  We may assume that $I=\NN$. For each $i\in\NN$, let 
$ M_i:= 2\, \max_{j\leq i} \, N_j.$

\smallskip Enumerate each set $A_i$ as an increasing sequence $(n_i(k))_{k\in\mathbb N}$ and for $s\geq 1,$ define
 \begin{align*}
  A_{i, s}&:=\left\{n_{i}(sk):\ k\in\NN\right\},\\
  \widetilde{A}_{i, s}&:=\bigl(A_{i,s}+[-M_i,M_i]\bigr)\cap \NN.
 \end{align*}
Then (by subadditivity of upper density)
\begin{align*}
 \udens(\widetilde{A}_{i,s})
 \leq \frac{2M_i+1}{s}\cdot
\end{align*}

Since all sets $A_{i,s}$ have positive lower density, it follows that one can construct by induction a sequence of positive integers $(s(i))_{i\in\NN}$ such that for all $i\in\mathbb N,$
 $s(i)\geq M_i$ and 
 \[\udens\left(\widetilde{A}_{i, s(i)}\right)\leq \min_{j<i}\frac{1}{4^{i-j}}\, \ldens\left(A_{j, s(j)}\right).\]
 We then set 
 \[B_i:=A_{i, s(i)}\backslash  \bigcup_{j>i}\widetilde{A}_{j,s(j)}\,,\]
 so that
 \begin{align*}
  \ldens(B_i)&\geq \ldens \left(A_{i,s(i)}\right)\left(1-\sum_{j>i}\frac 1{4^{j-i}}\right)
  >0.
 \end{align*}
 Moreover, $B_i$ is clearly contained in $A_i$, $\min(B_i)\geq s(i)\geq M_i$, the sets $B_i$ are pairwise disjoint, and if $(n,m)\in B_i\times B_j$
 with $n\neq m$ and $j\geq i$, then
 \begin{itemize}
 \item[\sbt] either $j=i$, in which case $|n-m|\geq s(i)\geq M_i\geq 2N_i=N_i+N_j$;
  \item[\sbt] or $j> i$, in which case  $|n-m|\geq M_j\geq N_i+N_j$ since $n\notin B_j+[-M_j,M_j]$.
 \end{itemize}
\end{proof}

\smallskip
We can now give the proof of Theorem \ref{thm:operatorfhc}.
\begin{proof}[Proof of Theorem \ref{thm:operatorfhc}] Let $(A_p)_{p\in\NN}$ be a sequence of subsets of $\NN$ with positive lower density, and let $(V_p)_{p\in\NN}$ be a sequence of non-empty open subsets of $X$. We have to find a vector $x\in X$ such that $\mathcal N_T(x,V_p)\cap A_p$ has positive lower density for all $p\in\NN$; and for that we follow the proof of \cite[Theorem 6.18]{BM09}. Let us fix an $F$-norm $\Vert\,\cdot\,\Vert$ defining the topology of $X$.
 
 \smallskip
 Let $\mathcal D$ be the dense set given by the Frequent Hypercyclicity Criterion.
 For each $p\geq 1$, choose a vector $x_p\in\mathcal D\cap V_p$ and let $\alpha_p>0$ be
 such that $B(x_p,3\alpha_p)\subset V_p$.
 Let also $(\veps_p)_{p\geq 1}$ be a summable sequence of positive real numbers
 such that
 for all $p\geq 1,$ 
 $$p\veps_p+\sum_{q>p+1}\veps_q<\alpha_p.$$
 By unconditional convergence of the series involved in the Frequent Hypercyclicity Criterion, for each $p\geq 1,$
 one can find a positive integer $N_p$ such that, for any set $F\subset\NN\cap[N_p,\infty),$
 $$\left\|\sum_{n\in F}T^n x_i\right\|+\left\|\sum_{n\in F}S^nx_i\right\|\leq {\veps_p}\quad\hbox{ for all }i\leq p.$$
Now, let $(B_p)_{p\in\NN}$ be the sequence of subsets of $\NN$ with positive lower density associated to $(A_p)$ and $(N_p)$ by Lemma \ref{prop:subsets}. The vector $x$
 we are looking for is defined by
 $$x:=
 \sum_{p=1}^{\infty}\sum_{n\in B_p}S^n x_p.$$
 First we note that $x$ is well defined.
 Indeed, each series $\sum_{n\in B_p}S^n x_p$ is convergent, and since $B_p\subset [N_p,\infty)$ for all $p$ we have
 $$\sum_{p\geq 1}\left\|\sum_{n\in B_p}S^n x_p\right\|\leq\sum_{p\geq 1}\veps_p<\infty.$$
 Let us fix $p\geq 1$ and $n\in B_p$: we show that $T^nx\in V_p$. By definition of $x$, we have
 \begin{align*}
  \|T^n x-x_p\|&\leq
  \sum_{q=1}^{\infty}\left\|\sum_{\substack{m\in B_q\\ m>n}}S^{m-n} x_q \right\|+\sum_{q=1}^{\infty}\left\|\sum_{\substack{m\in B_q\\  m<n}}T^{n-m} x_q \right\|.
 \end{align*}
 To estimate the first sum, we decompose it as 
 $$\sum_{q=1}^{p}\left\|\sum_{\substack{m\in B_q\\ m>n}}S^{m-n} x_q\right\|+\sum_{q=p+1}^{\infty}\left\|\sum_{\substack{m\in B_q\\ m>n}}S^{m-n} x_q\right\|.$$
 Since $n\in B_p,$ we know that $m-n>\max(N_p,N_q)$ whenever $m\in B_q$ and $m>n$. By the choice of the sequence $(N_p)$, it follows that
 $$\sum_{q=1}^{\infty}\left\|\sum_{\substack{m\in B_q\\ m>n}}S^{m-n} x_q\right\|\leq p\veps_p+\sum_{q=p+1}^{\infty}\veps_q<\alpha_p.$$
 Estimating the second sum in the same way, we conclude that
 $$\|T^n x-x_p\|< 3\alpha_p,$$
 so that $T^n(x)\in V_p.$
\end{proof}

As a direct consequence of Theorem \ref{thm:operatorfhc} and the fact that every frequently hypercyclic weighted backward shift on $\ell_p(\mathbb Z_+),$ satisfies the Frequent Hypercyclicity Criterion \cite{BAYRUZSA}, we obtain

\begin{corollary}\label{shiftlp}
A  weighted backward shift on $\ell_p(\mathbb Z_+),$ $ 1\leq p<\infty$ is frequently hypercyclic if and only if it is hereditarily frequently hypercyclic.
\end{corollary}

\begin{remark} It should be clear from the proof of Theorem \ref{thm:operatorfhc} that the Frequent Hypercyclicity Criterion implies hereditary $\mathcal F$-hypercyclicity for any Furstenberg family $\mathcal F$ satisfying Lemma \ref{prop:subsets}. For example, this holds true for the family of sets with positive upper density and the family of sets with positive Banach upper density, see \cite[Lemma 2.2]{MaMePu}. 
Observe that if $\mathcal F$ and $\mathcal F'$ are two Furstenberg families, the inclusion
$\mathcal F\subset\mathcal F'$ does not formally imply that hereditary $\mathcal F$-hypercyclicity is a stronger property than hereditary $\mathcal F'$-hypercyclicity.
\end{remark}

\section{Ergodic theory} 

\subsection{Results and general strategy} It is well known that if  $T\in\mathfrak L(X)$ and if one can find a $T$-$\,$invariant Borel probability measure $\mu$ on $X$ with full support with respect to which $T$ is an ergodic transformation, then $T$ is frequently hypercyclic. Let us recall the argument. Let $(V_p)_{p\in\NN}$ be a countable basis of open sets for $X$. Applying Birkhoff's pointwise ergodic theorem to the characteristic functions $\mathbf 1_{V_p}$, we obtain a sequence $(\Omega_p)$ of subsets of $X$ with $\mu(\Omega_p)=1$ such that
$$\frac1N\, \# \mathcal (N_T(x,V_p)\cap[0,N-1])=\frac 1N\sum_{n=0}^{N-1}\mathbf 1_{V_p}(T^n x)\xrightarrow{N\to\infty}\mu(V_p)>0\quad\textrm{ for every }x\in \Omega_p.$$
Hence, any $x\in \bigcap_{p\in\NN} \Omega_p$ is a frequently hypercyclic vector for $T$. 

Moreover, it is also well known (see e.g. \cite{BAYGRITAMS}, \cite[Chapter 5]{BM09}, \cite{BAYMATHERGOBEST}) that if $X$ is a complex Banach space and if an operator $T\in\mathfrak L(X)$ admits ``sufficiently many'' $\TT$-eigenvectors (\textit{i.e.} eigenvectors whose  associated eigenvalues have modulus $1$), then it is indeed possible to find an ergodic measure with full support for $T$. 
So, it may seem reasonable to expect that operators with sufficiently many $\TT$-eigenvectors are hereditarily frequently hypercyclic.

\smallskip
Now, if one wants to repeat the above argument to show that a given operator is hereditarily frequently hypercyclic, Birkhoff's ergodic theorem is not enough: what is needed is a pointwise convergence result for averages of quantities of the form $\mathbf 1_V(T^nx)$ not only along the whole sequence of integers, but in fact along any sequence $(n_k)$ with positive lower density. Specifically, we will use a theorem of Conze \cite{Con77}, which we state in the next section (Theorem \ref{thm:conze}). The assumptions of Conze's theorem are much stronger than merely asking that $T$ is an ergodic transformation; so we will need to impose a rather strong condition on the $\TT$-eigenvectors in order be able to conclude that our operator $T$ is indeed hereditarily frequently hypercyclic. 

\smallskip
Let us recall a few definitions. Assume that $X$ is a complex Banach space. If $T\in\mathfrak L(X)$, a \textbf{$\TT$-eigenvector field} for $T$ is any bounded map $E:\TT\to X$ such that $TE(z)=z E(z)$ for every $z\in\TT$. Given a positive Borel measure $\sigma$ on $\TT$, we say that a family of $\TT$-eigenvector fields $(E_i)_{i\in I}$ is \textbf{$\sigma$-spanning} if $\overline{\rm span}\,\bigl( E_i(z)\,:\, z\in \TT\setminus N\,,\, i\in I\bigr)=X$ for every Borel set $N\subset \TT$ such that $\sigma(N)=0$. Similarly, we say that \emph{the $\TT$-eigenvectors of $T$ are $\sigma$-spanning} if for every Borel set $N\subset\TT$ such that $\sigma(N)=0$, the eigenvectors of $T$ with eigenvalues in $\TT\setminus N$ span a dense subspace of $X$. It follows from \cite[Lemma 5.29]{BM09} that the $\TT$-eigenvectors of $T$ are $\sigma$-spanning if and only if there exists a $\sigma$-spanning countable family of Borel $\TT$-eigenvector fields for $T$. Our aim is to prove the following theorem.

\begin{theorem}\label{thm:ergodic} Let $X$ be a separable complex Banach space, and let $T\in\mathfrak L(X)$. 
 Assume that one can find a $\lambda$-spanning, finite or countably infinite family $(E_i)_{i\in I}$ of $\mathbb T$-eigenvector fields for $T$, where $\lambda$ is the Lebesgue measure on $\TT$. Moreover, assume that one of the following holds true.
 \begin{itemize}
 \item[\rm (a)] $X$ has type $2$;
 \item[\rm (b)] $X$ has type $p\in [1,2)$, and each eigenvector field $E_i$ is $\alpha_i$-H\"olderian for some $\alpha_i>1/p-1/2$. 
 \end{itemize}
Then  $T$ is hereditarily frequently hypercyclic.
\end{theorem}

\smallskip Our strategy for proving this theorem should be clear: we will show that under the above assumptions, one can find a Borel $T$-$\,$invariant measure $\mu$ on $X$ for which one can apply Conze's Theorem \ref{thm:conze} below to get hereditary frequent hypercyclicity. This will yield a more precise result, Theorem \ref{thm:ergodicbis}.

\smallskip In what follows, by a \emph{measure-preserving dynamical system} $(X, \mathcal B, \mu, T)$, we mean a pair consisting of a probability space $(X, \mathcal B, \mu)$ and a measurable map $T:X\to X$ such that $\mu\circ T^{-1}=\mu$. Note that here we are departing from our standing notation: $X$ is an abstract space, not necessarily a topological vector space, and hence $T$ is not necessarily a linear operator. This ambiguity is in fact intentional, and should cause no confusion.

Given a measure-preserving dynamical system $(X, \mathcal B, \mu, T)$, we denote by $U_T:L^2(X, \mathcal B, \mu)\to L^2(X, \mathcal B, \mu)$ the associated Koopman operator, which is defined by
\[ U_Tf:= f\circ T.\]
This is an isometry of $L^2(X, \mathcal B, \mu)$, and a unitary operator if $T$ is bijective and bimeasurable, in which case we say that $T$ is an \emph{automorphism} of $(X, \mathcal B, \mu)$, or that the measure-preserving dynamical system $(X, \mathcal B, \mu, T)$ is \emph{invertible}.
\par\smallskip
We stress a technical point: all probability spaces $(X, \mathcal B, \mu)$ under consideration will be assumed to be \textbf{standard Borel}, \textit{i.e.} the underlying measurable space $(X,\mathcal B)$ is isomorphic to$\bigl(Z, \mathcal B(Z)\bigr)$ for some Polish space $Z$, where $\mathcal B(Z)$ is the Borel $\sigma$-algebra of $Z$. And when $X$ is already a Polish space, we assume that $\mathcal B$ is the Borel $\sigma$-algebra of $X$. 

\subsection{Conze's theorem} 
Let us introduce some terminology. A unitary operator $U:H\to H$ acting on a Hilbert space $H$ 
is said to have {\bf Lebesgue spectrum} if there exists a family of vectors $(f_i)_{i\in I}$ in $H$
such that $\{U^n f_i:\ n\in\ZZ,\ i\in I\}$ is an orthonormal basis of $H$. Observe that if $H$ is separable,
the family $(f_i)$ has to be  finite or countably infinite. When there exists a countably infinite such family $(f_i)$, we say that $U$ has {\bf countable Lebesgue spectrum}. Finally, we say that an invertible measure-preserving dynamical system $(X,\mathcal B,\mu, T)$  has (countable) Lebesgue spectrum if the restriction of the Koopman operator $U_T$ to $ L^2_0(X,\mathcal B,\mu):=\bigl\{ f\in  L^2(X,\mathcal B,\mu)\,:\, \int_{X} f\, d\mu=0\bigr\}$ has (countable) Lebesgue spectrum. We may also say that $T$ itself has (countable) Lebesgue spectrum.
\smallskip

Conze's theorem from \cite{Con77} now reads as follows.

\begin{theorem}\label{thm:conze}
Let $(X,\mathcal B,\mu, T)$ be an invertible measure-preserving dynamical system, and assume that $T$ has Lebesgue spectrum. If $(n_k)_{k\geq 0}$ is an increasing sequence of integers with positive lower density then, for any  $f\in L^1(\mu),$ 
\[\frac 1N\sum_{k=0}^{N-1} f(T^{n_k}x)\xrightarrow{N\to\infty}\int_X f\, d\mu\quad\hbox{ for } \mu\hbox{-almost every }x\in X.\]
\end{theorem}

\smallskip Getting back to the case where $X$ is a separable Banach space, it is easy to check that if $T\in\mathfrak L(X)$ and if one can find a $T$-$\,$invariant measure with full support $\mu$ such that $(X,\mathcal B,\mu, T)$ satisfies the assumptions of Conze's theorem, then $T$ is hereditarily frequently hypercyclic. However, Conze's theorem is only stated for automorphisms, and in our context it would be rather restrictive to confine ourselves to the case of automorphisms.
We will get round this problem thanks to the notion of \textbf{factor}. Recall that a measure-preserving dynamical system $(X,\mathcal B,\mu,T)$ is a factor of a measure-preserving dynamical system  $(Y ,\mathcal C,\nu,S)$, or that $(Y ,\mathcal C,\nu,S)$ is an \textbf{extension} of $(X,\mathcal B,\mu,T)$,  
if (possibly after deleting two sets of measure $0$ in $X$ and $Y$) there exists a measurable map $\pi:Y\to X$ such that
 $\pi\circ S=T\circ\pi$ and $\mu=\nu\circ\pi^{-1}$.  Using suitable extensions, we will be able to apply Conze's theorem to general operators $T$ \textit{via} the following corollary.

\begin{corollary}\label{prop:conzelike}
 Let $T\in\mathfrak L(X)$ and assume that there exists a $T$-$\,$invariant probability measure $\mu$ on $X$ 
 with full support such that $(X,\mathcal B,\mu,T)$ is a factor of an invertible measure-preserving dynamical system $(Y,\mathcal C,\nu,S)$ with  Lebesgue spectrum. Then $T$ is hereditarily frequently hypercyclic. More precisely, given a countable family $(A_i)_{i\in I}$ of subsets of $\NN$ with positive lower density and a family $(V_i)_{i\in I}$ of non-empty open sets, $\mu$-almost every $x\in X$ is such that 
 $\mathcal N_T(x,V_i)\cap A_i$ has positive lower density for every $i\in I$.
\end{corollary}
\begin{proof} It is enough to prove the result for a single pair $(A,V)$, where $A\subset \NN$ has positive lower density and $V$ is a non-empty open set. 

\smallskip
Let $(n_k)_{k\geq 1}$ be the increasing enumeration of $A$. Let also $W:= \pi^{-1}(V)$, where $\pi:(Y ,\mathcal C,\nu,S)\to (X,\mathcal B,\mu,T)$ is a factor map, so that $\nu(W)=\mu(V)>0$, and let
\[\Omega:=\left\{y\in Y:\ \frac 1N\sum_{k=1}^N \mathbf 1_{{W}}(S^{n_{k}}y)\xrightarrow{N\to\infty}\mu(V)\right\}.\]
By Conze's theorem, we know that $\nu(\Omega)=1$. Since we are working with standard Borel spaces, it follows that the set $\pi(\Omega)$ is $\mu$-measurable (being an analytic set, it is universally measurable) and that $\mu\bigl( \pi(\Omega)\bigr)=1$. So it is enough to show that every $x\in \pi(\Omega)$ is such that $\ldens\bigl( \mathcal N_T(x,V)\cap A\bigr)>0$.

Let $x\in \pi(\Omega)$, and write $x$ as $x=\pi(y)$ for some $y\in\Omega$. By definition of $\Omega$, we know that the set $\{k\geq 1:\ S^{n_{k}}y\in{W}\}$ has positive lower density; enumerate it as an increasing sequence $(m_{k})_{k\geq 1}.$ Then 
$B:=\{n_{m_{k}}:\ k\geq 1\}$ 
has positive lower density, it is contained in $A,$ and $S^n y\in  \pi^{-1}(V)$ for all $n\in B$. This means that $T^n x\in V$ for all 
$n\in B,$ which concludes the proof. 
\end{proof}

\subsection{Natural extensions} 
In order to apply Corollary \ref{prop:conzelike}, we need a simple way of extending a given measure-preserving dynamical system $(X,\mathcal B,\mu,T)$ to a measure-preserving dynamical system $(\widetilde X,\widetilde{\mathcal B}, \widetilde\mu, \widetilde T)$ where $\widetilde T$ is an automorphism of $(\widetilde X,\widetilde{\mathcal B},\widetilde\mu)$. Fortunately, there is a canonical procedure doing precisely that. 
Consider the set 
\[\widetilde X:=\bigl\{(x_k)_{k\geq 0}\in X^{\ZZ_+}:\ T(x_{k+1})=x_k\textrm{ for all }k\geq 0\bigr\},\]
and for $k\geq 0,$ let $\pi_k:\widetilde X\to X$ denote the projection onto the $k$-th coordinate of $\widetilde X.$ Endow $\widetilde X$ with the smallest $\sigma$-algebra $\widetilde{\mathcal B}$ 
which makes every projection $\pi_k$ measurable. The {\bf Rokhlin's natural extension} of $T$ is the measurable transformation $\widetilde T:\widetilde X\to\widetilde X$ defined by
$$\widetilde T(x_0, x_1, \dots):=(T(x_0),x_0,x_1,\dots).$$
One can prove that there exists a unique probability measure $\widetilde \mu$ on $(\widetilde X,\widetilde{\mathcal B})$ such that $\widetilde\mu\circ\pi_k^{-1}=\mu$ for all $k\geq 0$. (Here, the fact $(X,\mathcal B)$ is a standard Borel space is needed.) 
Then, it is a simple exercise to check that $\widetilde T:\widetilde X\to\widetilde X$ is an automorphism of  $(\widetilde X,\widetilde B, \widetilde\mu)$ such that  $\pi_k\circ\widetilde T=T\circ\pi_k$ for all $k\geq 0$. In particular,
$(X,\mathcal B,\mu,T)$ is a factor of $(\widetilde{ X},\widetilde{ \mathcal B},\widetilde \mu,\widetilde T)$ as witnessed by $\pi_0: \widetilde X\to X$. Note also that if $X$ is a Polish space, then $\widetilde X$ is a Borel subset of the Polish space $X^{\ZZ_+}$ endowed with the product topology, and $\widetilde{\mathcal B}$ is the Borel sigma-algebra of $\widetilde X$. For more details, see e.g. \cite[Section 8.4]{URM} or \cite{NP11}.

\smallskip
We will need the following lemma. 
 Recall that if $X$ is a complex Fr\'echet space, a Borel probability measure $\mu$ on $X$ is said to be \textbf{Gaussian} if every continuous linear functional on $X$ has a complex symmetric Gaussian distribution when considered as a random variable on 
$(X,\mathcal B,\mu)$.
\begin{lemma}\label{tildegauss} Assume that $X$ is a complex Fr\'echet space, that $T\in\mathfrak L(X)$, and that the measure $\mu$ is Gaussian. Then $\widetilde X$ is a Fr\'echet space when endowed with the induced product topology of $X^{{\mathbb{Z}}^{+}}$, $\widetilde T$ is an invertible continuous linear operator on $\widetilde X$, and the measure $\widetilde\mu$ is Gaussian.
\end{lemma}
\begin{proof} It is clear that $\widetilde X$ is a closed linear subspace of $X^{\ZZ_+}$ (and hence a Fr\'echet space), and that $\widetilde T$ is an invertible continuous linear operator on $\widetilde X$. 

Let $\phi$ be a continuous linear functional on $\widetilde X$. By the Hahn-Banach theorem, one can extend $\phi$ to a continuous linear functional $\Phi$ on $X^{\ZZ_+}$. Since $X^{\ZZ_+}$ is endowed with the product topology, this linear functional $\Phi$ has the form
\[ \Phi=\sum_{k=0}^N x^*_k\circ \pi_k,\]
where $x_0^*,\dots , x_N^*\in X^*$. Moreover, by definition of $\widetilde X$, we have $\pi_k=T^{N-k}\circ \pi_N$ on $\widetilde X$ for $k=0,\dots ,N$; so we may write 
\[ \phi =x^*\circ \pi_N\quad \hbox{where}\quad  x^*=\sum_{k=0}^N x_k^*\circ T^{N-k}\in X^*.\]

Hence, $\widetilde \mu\circ \phi^{-1}=\bigl( \widetilde\mu\circ\pi_N^{-1}\bigr)\circ (x^*)^{-1}=\mu\circ (x^*)^{-1}$. Since $\mu$ is a Gaussian measure, it follows that $\widetilde\mu$ is Gaussian as well.
\end{proof}

\subsection{Two facts concerning unitary operators}

We will need several results on unitary operators. These results are certainly well known, but since they are rather difficult to locate in the literature, we provide some details.

\smallskip
Let $U$ be a unitary operator acting on a complex separable Hilbert space $H.$  By Herglotz's theorem, for any $f\in H$, there exists 
a unique positive and finite Borel measure $\sigma_f$ on $\TT$, called the {\bf spectral measure} of $f$ with respect to $U$, such that
\[ \forall n\in\ZZ\;:\; \langle U^n f,f\rangle=\widehat{\sigma_f}(n).\]

\smallskip
Note in particular that $\sigma_f$ is equal to the Lebesgue measure $\lambda$ if and only if the sequence $(U^nf)_{n\in\ZZ}$ is orthonormal. This explains 
the terminology ``Lebesgue spectrum''.

\smallskip Denote by $C(f)$ the \emph{cyclic subspace} generated by $f,$
\textit{i.e.}  $C(f)=\overline{\textrm{span}}(U^n f:\ n\in\mathbb Z)$. With this notation, one form of the Spectral Theorem reads as follows (see e.g. \cite[Chapter IX]{Conw}, \cite[Appendix, Section 2]{Parry} or \cite{Martine}): there exists
a finite or infinite sequence of vectors
$f_i\in H,$ $0\leq i<m$, where $m\in\mathbb N\cup\{\infty\}$, such that 
$$H=\bigoplus_{0\leq i< m} C(f_i)\quad\textrm{ and }\quad\sigma_{f_0}\gg\sigma_{f_1}\gg\cdots\gg\sigma_{f_i}\gg\cdots.$$
Moreover, these measures are essentially unique in the following sense: for any other sequence $(g_i)_{0\leq i<m'}$ 
satisfying $H=\bigoplus_{0\leq i<m'}C(g_i)$ and $\sigma_{g_0}\gg\sigma_{g_1}\gg\cdots\gg\sigma_{g_i}\gg\cdots,$
we have $m'=m$ and $\sigma_{f_i}\sim\sigma_{g_i}$ for all $i.$ The {\bf maximal spectral type} of $T$ is then defined as (the equivalence class of ) the measure $\sigma_{f_0}$.

\smallskip Observe that for any $f\in H,$ the restriction of $U$ to $C(f)$ is unitary equivalent to the multiplication operator $M_{\sigma_f}:L^2(\TT,\sigma_f)\to L^2(\TT,\sigma_f)$
defined by $M_{\sigma_f}u (z):= zu(z)$, $z\in\mathbb T$. 

\smallskip
Suppose now that $U$ has countable Lebesgue spectrum and let $(f_i)_{i\in \NN}$ be a sequence such that 
$\{U^n f_i:\ n\in\mathbb Z,\ i\in\NN\}$ is an orthonormal basis of $H.$ Then 
$H=\bigoplus_{i\in \NN}C(f_i)$, and $\sigma_{f_i}=\lambda$ for all $i\in I$.
Hence, 
$U$ is unitarily equivalent to $M_{\lambda}^{(\infty)}:=\bigoplus_{i=1}^\infty M_\lambda$ acting on $\bigoplus_{i=1}^\infty L^2(\TT,\lambda)$.
Conversely, it is clear that if $U\cong M_{\lambda}^{(\infty)}$ then  $U$ has countable Lebesgue spectrum.

\smallskip
By the uniqueness part of the Spectral Theorem, it follows in particular that the maximal spectral type of a unitary operator with countable Lebesgue spectrum is the Lebesgue measure. The next  lemma is a kind of converse.

\begin{lemma}\label{lem:lebesguespectrum}
 Let $U$ be a unitary operator on a complex separable Hilbert space $H$ satisfying the following two conditions.
 \begin{itemize}
  \item[(a)] There exists a closed subspace $K\subset H$ such that $U(K)=K$ and $U_{| K}:K\to K$ has countable Lebesgue spectrum.
  \item[(b)] The maximal spectral type of $U$ is the Lebesgue measure.
 \end{itemize}
 Then $U:H\to H$ has countable Lebesgue spectrum.
\end{lemma}
\begin{proof} We will use the following notation: if $\sigma$ is a positive finite Borel measure on $\TT$ and $N\in\NN\cup\{\infty\}$, we denote by  $M_{\sigma}^{(N)}$ the operator $\bigoplus_{0\leq j<N} M_\sigma$
acting on $\bigoplus_{0\leq j<N} L^2(\TT,\sigma)$. Recall also that $\lambda$ is the Lebesgue measure on $\TT$.

By (a), we know that 
\[ U_{| K}\cong M_{\lambda}^{(\infty)}.\]
 
 Consider now the operator $U_{| K^\perp}:K^\perp\to K^\perp$. By (b), its maximal spectral type is absolutely continuous with respect to the Lebesgue measure $\lambda$. 
 By the ``second formulation'' of the Spectral Theorem (see e.g. \cite{Martine} or \cite[Chapter IX]{Conw}),
 there exist pairwise disjoint Borel sets $\Delta_\infty,\Delta_1,\Delta_2,\dots$ and positive finite measures $\mu_\infty,\mu_1,\mu_2,\dots$ 
 on $\TT$ supported on $\Delta_\infty,\Delta_1,\Delta_2,\dots$ and absolutely continuous with respect to $\lambda$, such that $U_{|K^\perp}$ is unitarity equivalent to 
 $M_{\mu_\infty}^{(\infty)}\oplus M_{\mu_1}^{(1)}\oplus M_{\mu_2}^{(2)}\oplus\cdots.$ Note that some of the measures $\mu_n$ may be $0$. For $n\in\{ \infty\}\cup\NN$, we may write $\mu_n=f_n\, \lambda$ for some nonnegative function $f_n\in L^1(\TT)$, and we may assume that 
$\Delta_n=\{ f_n>0\}$. Hence, if we set $\nu_n:=\mathbf{1}_{\Delta_n}\, \lambda,$
we get $M_{\mu_n}\cong M_{\nu_n}$ since the measures $\mu_n$ and $\nu_n$ are equivalent.
Therefore, \[ U_{|K^\perp}\cong M_{\nu_\infty}^{(\infty)}\oplus M_{\nu_1}^{(1)}\oplus M_{\nu_2}^{(2)}\oplus\cdots.\]

Now, let $E:=\TT\backslash\left(\Delta_\infty\cup \bigcup_{n\geq 1}\Delta_n\right)$ and $\nu:=\mathbf{1}_{E}\, \lambda.$ Then 
$\lambda=\nu+\nu_\infty+\sum\limits_{n\in\NN}\nu_n$ so that 
\[ M_\lambda\cong M_\nu\oplus M_{\nu_\infty}\oplus M_{\nu_1}\oplus \cdots.\]

Therefore, we obtain
\begin{align*}
 U&\cong U_{|K}\oplus U_{| K^\perp}\\
 &\cong M_\nu^{(\infty)}\oplus M_{\nu_\infty}^{(\infty)}\oplus M_{\nu_1}^{(\infty)}\oplus M_{\nu_2}^{(\infty)}\oplus \cdots\\
 &\qquad\qquad \oplus M_{\nu_\infty}^{(\infty)}\oplus M_{\nu_1}^{(1)}\oplus M_{\nu_2}^{(2)}\oplus\cdots\\
 &\cong M_\nu^{(\infty)}  \oplus M_{\nu_\infty}^{(\infty)}\oplus M_{\nu_1}^{(\infty)}\oplus M_{\nu_2}^{(\infty)} \cdots\\
 &\cong M_\lambda^{(\infty)},
\end{align*}
which means that $U$ has countable Lebesgue spectrum.
\end{proof}

\smallskip The following observation is useful in order to check assumption (a) in Lemma \ref{lem:lebesguespectrum}.

\begin{lemma}\label{passitriv} Let $(f_i)_{i\in I}$ be a countably infinite family of  vectors in $H$. Assume that the cyclic subspaces $C(f_i)$ are pairwise orthogonal, and let $K:=\oplus_{i\in I} C(f_i)$. If $\sigma_{f_i}$ is equivalent to the Lebesgue measure for all $i\in I$, then $U_{| K}$ has countable Lebesgue spectrum.
\end{lemma}

\begin{proof} For each $i\in I$, there exists $g_i\in C(f_i)$ such that $\sigma_{g_i}$ is exactly equal to Lebesgue measure, and for any such $g_i$ we have $C(g_i)=C(f_i)$; see e.g. \cite[pp. 93--94]{Parry}. Hence, $\{U^ng_i\,:\, i\in I, n\in\ZZ\}$ is an orthonormal basis of $K$.
\end{proof}

We will also need the following fact.

\begin{lemma}\label{lem:densemaximalspectrum}
  Let $U$ be a unitary operator acting on a separable Hilbert space $H$,  and let $\sigma$ be a finite, positive Borel measure on $\TT$. If there exists a dense set $\mathcal D\subset H$ such that 
 $\sigma_f\ll \sigma$ for all $f\in \mathcal D$, then $\sigma_f\ll \sigma$ for all $f\in H$.
\end{lemma}
\begin{proof} Denote by $M(\TT)$ the space of all complex Borel measures on $\TT$. By \cite[Corollary 2.2]{Martine}, the map $f\mapsto \sigma_f$ is continuous from $H$ into $M(\TT)$ endowed with the \emph{norm} topology. The result follows immediately.
\end{proof}

\subsection{Gaussian measures and countable Lebesgue spectrum} In this section, we prove a general result about Gaussian linear measure-preserving dynamical systems. This result looks very much like \cite[Theorem 14.3.2]{CFS82}, but it is not clear to us that it can be formally deduced from it; so we give a more or less complete proof.

\begin{proposition}\label{CFSlike} Let $(X,\mathcal B,\mu, T)$ be a measure-preserving dynamical system where $X$ is a separable Fr\'echet space, T is an invertible linear operator and $\mu$ is a Gaussian measure on $X$. Assume that for any $x^*\in X^*\subset L^2(\mu)$, the spectral measure $\sigma_{x^*}$ with respect to $U_T$ is absolutely continuous with respect to Lebesgue measure. Then $(X,\mathcal B,\mu, T)$ has countable Lebesgue spectrum.
\end{proposition}

\smallskip For the proof, we need to recall some basic facts concerning the $L^2$-$\,$space of a Gaussian measure. In what follows, $\mu$ is a Gaussian measure on the (separable) complex Fr\'echet space $X$.

\smallskip
Let $\mathcal G:=\overline{\textrm{span}}\,\bigl(\langle x^*,\cdot\rangle:\ x^*\in X^*\bigr)$, 
where the closure is taken in $L^2(X,\mathcal B,\mu)$. This is a Gaussian subspace of $L^2(X,\mathcal B,\mu)$, in the sense that any function
in $\mathcal G$ has symmetric complex gaussian distribution. Moreover, $\mathcal G$ is clearly $U_T$-invariant.

For $k\geq 0,$ let us denote by $\mathcal G^k$ the space of homogeneous polynomials of degree $k$ in the elements of $\mathcal G$, 
with $\mathcal G^0=\mathbb C.$ The subspaces $\mathcal G^k$, $k\geq 0$, are linearly independent (which is not obvious), and one can orthonormalize them
thanks to the so-called \emph{Wick transform} $f\mapsto\; \wick{f}$ which is defined on $ \bigcup_{k\geq 0} \mathcal G^k$ as follows:
\[ \wick{f}\;=\left\{\begin{array}{ll} f&\hbox{if $f$ is constant}\\
f- P_k f\quad&\hbox{if $f\in\mathcal G^k$, $k\geq 1$}
\end{array}
\right.
\]
where we denote by $ P_k$  the orthogonal
 projection onto $\overline{\textrm{span}}\,\bigl(\mathcal G^i:0\leq i\leq k-1\big).$

 With this notation, we have the orthogonal decomposition
 \[ L^2(X,\mathcal B,\mu)=\bigoplus_{k=0}^{\infty}\overline{\wick{\mathcal G^k}}\]
 
 Moreover, for any $f_1,\dots ,f_k, g_1,\dots ,g_k\in\mathcal G$, the scalar product $ \big\langle \wick{f_1\cdots f_k}\; ,\, \wick{g_1\dots g_k}\big\rangle$ is given by
 \begin{equation}\label{scalprod} \big\langle \wick{f_1\cdots f_k}\; ,\, \wick{g_1\dots g_k}\big\rangle=\sum_{\mathfrak s\in\mathfrak S_k}\langle f_{\mathfrak s(1)},g_1\rangle\cdots \langle  f_{\mathfrak s(k)},g_k\rangle,\end{equation}
 where $\mathfrak S_k$ is the permutation group of $\{ 1,\dots ,k\}$. 
 
 \smallskip
For details on these general facts, we refer to \cite[Chapter 8]{Peller} or \cite[Chapters 2 and 3]{J}.  We will also need the following lemma (see \cite[Lemma 3.28]{BAYGRITAMS}).

\begin{lemma}\label{lem:UTGaussian} Let  $T\in\mathfrak L(X)$, and let $\mu$ be a $T$-$\,$invariant Gaussian measure on $X$.
For any $f_1,\dots ,f_k\in\mathcal G$, we have
\[ U_T\bigl(\wick{f_1\cdots f_k}\bigr)=\; \wick{(U_Tf_1)\cdots (U_Tf_k)}\]
In particular, each subspace $\wick{\mathcal G^k}$ is $U_T$-invariant.
\end{lemma}

\smallskip We can now give the proof of Proposition \ref{CFSlike}.
\begin{proof}[Proof of Proposition \ref{CFSlike}] Let us first recall that for any $f,g\in L^2(\mu)$, there is a unique complex Borel measure $\sigma_{f,g}$ on $\TT$ with Fourier coefficients 
\[ \widehat{\sigma_{f,g}}(n)=\langle U_T^n f,g\rangle,\quad n\in\mathbb Z.\]
 If $f=g$, then $\sigma_{f,f}$ is the spectral measure $\sigma_f$; and the existence of $\sigma_{f,g}$ for arbitrary functions $f$ and $g$ follows from a polarization argument. Indeed, we have 
 \[ \sigma_{f,g}= \frac14\sum_{k=0}^3 \mathbf i^k\sigma_{f+\mathbf i^k g}.\]
 
 This formula and the assumption of the proposition show that $\sigma_{x^*, y^*}$ is absolutely continuous with respect to Lebesgue measure for any $x^*, y^*\in X^*$. Note also that the map $(f,g)\mapsto \sigma_{f,g}$ is obviously $\RR$-bilinear.
 
 \smallskip
In what follows, we denote by $\star$ the convolution product for measures on $\TT$, and we write the elements of $X^*$ as $f,g, \dots$ rather than $x^*, y^*, \dots$ to avoid the proliferation of stars.

\smallskip
\begin{fact}\label{convol} If $f_1, \dots ,f_k\in X^*$, then $\displaystyle \sigma_{\wick{f_1\cdots f_k}}= \sum_{\mathfrak s\in\mathfrak S_k} \sigma_{f_{\mathfrak s(1)}, f_1}\star\cdots \star \sigma_{f_{\mathfrak s(k)},f_k}$.
\end{fact}
\begin{proof}[Proof of Fact \ref{convol}] By Lemma \ref{lem:UTGaussian} and (\ref{scalprod}), we have for all $n\in\ZZ$:
\begin{align*}
\widehat{\sigma}_{\wick{f_1\cdots f_k}}(n)&=\big\langle \wick{(U_T^nf_1)\cdots (U_T^nf_k)}\, ,\, \wick{f_1\cdots f_k}\big\rangle\\
&=\sum_{\mathfrak s\in\mathfrak S_k} \langle U_T^nf_{\mathfrak s(1)},f_1\rangle\cdots \langle  U_T^nf_{\mathfrak s(k)},f_k\rangle\\
&= \sum_{\mathfrak s\in\mathfrak S_k} \widehat\sigma_{f_{\mathfrak s(1)}, f_1}(n)\cdots \widehat\sigma_{f_{\mathfrak s(k)}, f_k}(n).
\end{align*}
\end{proof}

\smallskip Now, let us denote by $\mathcal D$ the set of all functions $f\in L^2(\mu)$ of the form
\[ f=\sum_{k=1}^N \wick{\sum_{j=1}^{m_k}f_{j,1}\cdots f_{j,k}}\quad\hbox{with $m_k\geq 1$ and $f_{j,i}\in X^*$}.\]

This is a dense linear subspace of $L^2_0(\mu)$.

\begin{fact}\label{trucspectral} If $f\in \mathcal D$, then $\sigma_f$ is absolutely continuous with respect to Lebesgue measure.
\end{fact}
\begin{proof}[Proof of Fact \ref{trucspectral}]

Let $f\in \mathcal D$, so that $\displaystyle f=\sum_{k=1}^N f_k$ where 
\[ f_k= \sum_{j=1}^{m_k}\wick{f_{j,1}\cdots f_{j,k}}\quad\hbox{with $f_{j,i}\in X^*$}, \quad\hbox{so that } f_k\in {\wick{\mathcal G^k}}.\]
By orthogonality and $U_T$-invariance of the subspaces $\wick{\mathcal G^k}$ we have 
for all $n\in\ZZ$:
$$\widehat{\sigma_f}(n)=\langle U_T^n f,f\rangle=\sum_{k=1}^N \langle U_T^n f_k,f_k\rangle=\sum_{k=1}^N \widehat{\,\sigma_{f_k}}(n),$$
so that $\sigma_f= \sigma_{f_1}+\cdots+\sigma_{f_N}$. Hence, it is enough to check that each measure $\sigma_{f_k}$ is absolutely continuous with respect to Lebesgue measure. But this is clear by Fact \ref{convol} and bilinearity of the map $(f,g)\mapsto \sigma_{f,g}$: indeed, we have
\begin{align*} \sigma_{f_k}= \sum_{j,j'=1}^N\sigma_{ \wick{f_{j,1}\cdots f_{j,k}}\,,\, \wick{f_{j',1}\cdots f_{j',k}}}=\sum_{j,j'=1}^N \sum_{\mathfrak s\in\mathfrak S_k} \sigma_{f_{j,\mathfrak s(1)},f_{j',1}}\star\cdots\star \sigma_{f_{j,\mathfrak s(k)},f_{j',k}},
\end{align*}
and the result follows since all measures  $\sigma_{f_{j,\mathfrak s(i)},f_{j',i}}$ are absolutely continuous with respect to Lebesgue measure.
\end{proof}

\smallskip
We can now  finish the proof of Proposition \ref{CFSlike}.
By Fact \ref{trucspectral} and Lemma \ref{lem:densemaximalspectrum}, 
the maximal spectral type of $U_{T}$ is absolutely continuous with respect to the Lebesgue measure.
Now, take any $f\in X^*\setminus\{ 0\}$. Then $\sigma_f$ is a non-zero measure which is absolutely continuous
with respect to Lebesgue measure. It follows that 
there exists $k_0\geq 1$
such that, for all $k\geq k_0,$ the measure $\sigma_{:f^k:}=k!\, \sigma_f\star\cdots\star\sigma_f$ ($k$ times)
is \emph{equivalent} to Lebesgue measure (see e.g. \cite[Lemma 3.6]{NP}, where the symmetry assumption on the measure is in fact not necessary). Let us set
$$K:=\overline{\textrm{span}}\,\{U_{T}^n (\wick{f^k})\, :\, \ n\in\mathbb Z,\ k\geq k_0\}.$$
By orthogonality and $U_T$-invariance of the spaces $\wick{\mathcal G^k}$, and applying Lemma \ref{passitriv}, we see that $(U_T)_{| K}$ has countable Lebesgue spectrum.  Hence, by Lemma \ref{lem:lebesguespectrum}, $U_{T}$ has countable Lebesgue spectrum.
\end{proof}

\smallskip We now explore the consequences of Proposition \ref{CFSlike} for not necessarily invertible Gaussian linear measure-preserving dynamical systems. Let us first recall that spectral measures exist for every measure-preserving dynamical system, invertible or not:  if $(X,\mathcal B, \mu, T)$ is a measure-preserving dynamical system then, for any $f\in L^2(\mu)$, there is a unique positive Borel measure $\sigma_{f}$ on $\TT$ with nonnegative Fourier coefficients 
\[ \widehat{\,\sigma_{f}}(n)=\langle U_T^n f,f\rangle_{L^2(\mu)},\quad n\geq 0.\]

\begin{corollary}\label{celuila} Let $(X,\mathcal B,\mu, T)$ be a measure-preserving dynamical system, where $X$ is a complex separable Fr\'echet space, T is a continuous linear operator and $\mu$ is a Gaussian measure on $X$. Assume that for any $x^*\in X^*\subset L^2(\mu)$, the spectral measure $\sigma_{x^*}$ with respect to $U_T$ is absolutely continuous with respect to Lebesgue measure. Then, the natural extension of  $(X,\mathcal B,\mu, T)$ has countable Lebesgue spectrum.
\end{corollary}
\begin{proof} Let $(\widetilde X,\widetilde{\mathcal B},\widetilde{\mu}, \widetilde{T})$ be the natural extension. By Lemma \ref{tildegauss}, we know that $\widetilde T$ is an invertible linear operator and $\widetilde\mu$ is a Gaussian measure. Hence, by Proposition \ref{CFSlike}, it is enough to show that for any continuous linear functional $\phi$ on $\widetilde X$, the spectral measure $\sigma_{\phi}$ (with respect to $U_{\widetilde T}$) is absolutely continuous with respect to Lebesgue measure. 

By the proof of Lemma \ref{tildegauss}, one can find $N\in\ZZ_+$ and $x^*\in X^*$ such that $\phi =x^*\circ \pi_N$ on $\widetilde X$, where $\pi_N:X^{\ZZ_+}\to X$ is the projection onto the $N$-th coordinate. Now, we apply the following purely formal fact, which holds true for any measure-preserving dynamical system $(X,\mathcal B,\mu, T)$.
\begin{fact}\label{formal} Let $f\in L^2(X,\mathcal B,\mu)$ and $F:=f\circ\pi_N\in L^2(\widetilde X, \widetilde B, \widetilde \mu)$. Then the spectral measure $\sigma_{F}$ with respect to $U_{\widetilde T}$ is equal to $\sigma_f$, the spectral measure of $f$ with respect to $U_T$.
\end{fact} 
\begin{proof}[Proof of Fact \ref{formal}] For all $n\geq 0$, 
\begin{align*}
\langle U_{\widetilde T}^n F , F\rangle_{L^2(\widetilde\mu)}=\langle F \circ\widetilde T^n,F\rangle_{L^2(\widetilde \mu)}&=\langle f\circ\pi_N\circ\widetilde T^n,f\circ\pi_N\rangle_{L^2(\widetilde \mu)}\\
&=\langle f\circ T^n \circ\pi_N,f\circ\pi_N\rangle_{L^2(\widetilde \mu)}\\
&=\langle f\circ T^n ,f\rangle_{L^2(\mu)}\\&= \langle U_T^n f, f\rangle_{L^2(\mu)}.
\end{align*}
\end{proof}

By Fact \ref{formal}, $\sigma_\phi=\sigma_{x^*}$ is absolutely continous with respect to Lebesgue measure, which concludes the proof of Corollary \ref{celuila}.
\end{proof}

\subsection{Proof of Theorem \ref{thm:ergodicbis}} We can now state and prove very easily the following more precise version of Theorem \ref{thm:ergodic}.
\begin{theorem}\label{thm:ergodicbis} Let $X$ be a separable complex Banach space, and let $T\in\mathfrak L(X)$. 
 Assume that one can find a $\lambda$-spanning, finite or countably infinite family $(E_i)_{i\in I}$ of $\mathbb T$-eigenvector fields for $T$, where $\lambda$ is the Lebesgue measure on $\TT$. Moreover, assume that one of the following holds true.
 \begin{itemize}
 \item[\rm (a)] $X$ has type $2$;
 \item[\rm (b)] $X$ has type $p\in [1,2)$, and each $E_i$ is $\alpha_i$-H\"olderian for some $\alpha_i>1/p-1/2$. 
 \end{itemize}
Then there exists a $T$-$\,$invariant Gaussian measure $\mu$ on $X$ with full support such that $(X,\mathcal B,\mu,T)$ is a factor of 
a measure-preserving system which has countable Lebesgue spectrum. In particular, $T$ is hereditarily frequently hypercyclic; and more precisely: given a countable family $(A_i)_{i\in I}$ of subsets of $\NN$ with positive lower density and a family $(V_i)_{i\in I}$ of non-empty open sets, $\mu$-almost every $x\in X$ is such that 
 $\mathcal N_T(x,V_i)\cap A_i$ has positive lower density for every $i\in I$.
\end{theorem}

\begin{proof} Under the  assumptions above, there exists a $T$-$\,$invariant Gaussian measure $\mu$ on $X$ with full support 
such that, for all $x^*\in X^*,$ the measure $\sigma_{x^*}$ is absolutely continuous with respect to Lebesgue measure:  
this is contained for instance in  \cite[Lemma 5.35, (4)]{BM09}. So, the result follows immediately from Corollary \ref{celuila} and Corollary \ref{prop:conzelike}.
\end{proof}

\subsection{The Frequent Hypercyclicity Criterion again} The tools introduced in the previous sections allow us to give another proof of Theorem \ref{thm:operatorfhc}. Arguably, this proof is much less elementary. Let us say that an operator $T\in\mathfrak L(X)$ \textbf{has countable Lebesgue spectrum after extension} if there exists a $T$-invariant Borel probability measure $\mu$ on $X$ with full support such that $(X,\mathcal B,\mu,T)$ 
is a factor of a measure-preserving dynamical system which has countable Lebesgue spectrum. By Corollary \ref{prop:conzelike}, any such operator $T$ is hereditarily frequently hypercyclic. 

\begin{proposition}\label{thm:operatorfhcergodic}
 Let $T\in\mathfrak L(X)$ be an operator satisfying the Frequent Hypercyclicity Criterion.
Then, $T$ has countable Lebesgue spectrum after extension.
\end{proposition}
\begin{proof}
 It is shown in \cite{MuPe13} that there exists a $T$-invariant measure $\mu$ on $X$ with full support such that $(X,\mathcal B,T,\mu)$
 is a factor of a Bernoulli shift; and it is well known that  Bernoulli shifts have countable Lebesgue spectrum (see e.g. \cite[Theorem 4.30 and Theorem 4.33]{Wa}).
\end{proof}

\smallskip One can also prove the following ``probabilistic'' version of Proposition \ref{thm:operatorfhcergodic}. Let us say that a sequence $(x_n)_{n\in\ZZ}$ is a \emph{bilateral backward orbit} for an operator $T$ if $Tx_n=x_{n-1}$ for all $n\in\ZZ$. 
\begin{proposition}\label{backward} Let $X$ be a complex Fr\'echet space, and let $T\in\mathcal L(X)$. Assume that there exists a bilateral backward orbit $(x_n)_{n\in\ZZ}$ for $T$ such that 
$\overline{\rm span}\bigl( x_n\,:\, n\in\ZZ\bigr)=X$ and the series $\sum g_n x_n$ is almost surely convergent, where $(g_n)$ is a sequence of independent complex standard Gaussian variables. Then $T$ has countable Lebesgue spectrum after extension. More precisely, there exists a $T$-invariant Gaussian measure $\mu$ on $X$  with full support such that $(X,\mathcal B, \mu, T)$ is a factor of a measure-preserving dynamical system with countable Lebesgue spectrum.
\end{proposition}
 The (almost sure) convergence of the bilateral series $\sum g_n x_n$ means that both series $$\sum_{n\geq 0} g_n x_n \quad \textrm{ and }\quad\sum_{n<0} g_n x_n$$ are (almost surely) convergent.

\smallskip
\begin{proof} Let $\mu$ be the distribution of the random variable $\xi:=\sum_{n\in\ZZ} g_n x_n$. This is a Gaussian measure, which has full support since  $\overline{\rm span}\bigl( x_n\,:\, n\in\ZZ\bigr)=X$, and which is $T$-invariant because $(x_n)$ is a bilateral backward orbit for $T$. By Corollary \ref{celuila}, it is enough to show that for any $x^*\in X^*\subset L^2(\mu)$, the spectral measure $\sigma_{x^*}$ of $x^*$ with respect to $U_T$ is absolutely continuous with respect to Lebesgue measure.

By orthogonality of the Gaussian variables $g_k$, we have for all $n\geq 0$:
\begin{align*}
\widehat{\sigma_{x^*}}(n)=\langle U_T^n x^*, x^*\rangle&=\sum_{k\in\ZZ} \langle x^*, T^n x_k\rangle\,\overline{\langle x^*, x_k\rangle}\\
&=\sum_{k\in\ZZ} \langle x^*, x_{k-n}\rangle\,\overline{\langle x^*, x_k\rangle}.
\end{align*}

The series is absolutely convergent since the almost sure convergence of the scalar Gaussian series $\sum \langle x^*, x_k\rangle \, g_k$ implies that $\sum_{k\in\ZZ} \vert \langle x^*, x_k\rangle\vert^2<\infty$.

Now, let $\varphi\in L^2(\TT)$ be the function with Fourier coefficients $\widehat \varphi(k):= \langle x^*, x_{-k}\rangle$, $k\in\mathbb Z$, i.e.
\[ \varphi(z)\sim \sum_{k\in\ZZ} \langle x^*, x_{-k}\rangle \,z^k;\]
and let $g:=\vert\varphi\vert^2\in L^1(\TT)$. By definition, we have for all $n\geq 0$:
\[ \widehat g(n)=\sum_{k\in\ZZ} \widehat{\overline\varphi}(k)\,\widehat{\varphi}(n-k)=\sum_{k\in\ZZ} \overline{\langle x^*, x_k\rangle}\, \langle x^*, x_{k-n}\rangle.\]

Since two positive measures on $\TT$ with the same non-negative Fourier coefficients must be equal, it follows that $\sigma_{x^*}= g(\lambda)\, d\lambda$, which concludes the proof.
\end{proof}

\begin{remark} The above proof shows in particular that if $(x_n)_{n\in\ZZ}$ is a {bilateral backward orbit} for $T$ such that the series $\sum g_n x_n$ is almost surely convergent, then the distribution of the random variable $\xi:=\sum_{n\in\ZZ} g_n x_n$ is a strongly mixing measure for $T$. This is not specific to gaussian variables: as shown in \cite{Kevin}, the same result holds true if $(g_n)$ is replaced by any sequence of independent, identically distributed random variables.
\end{remark}

\smallskip We mentioned above that Proposition \ref{backward} is a probabilistic version of Proposition \ref{thm:operatorfhcergodic}; let us be a little bit more explicit. The following fact (which was observed independently by A. L\'opez-Mart\'{\i}nez) can be extracted from the proof of \cite[Theorem 4.9]{Kevin}.

\begin{fact}\label{Kevin} Let $X$ be a Fr\'echet space. If $T\in\mathfrak L(X)$ satisfies the Frequent Hypercyclicity Criterion, then there exists a bilateral backward orbit $(x_n)$ for $T$ such that $\overline{\rm span}\bigl( x_n\,:\, n\in\ZZ\bigr)=X$ and the series $\sum x_n$ is unconditionally convergent. 
\end{fact}

In view of that, the next result is an improvement of Proposition \ref{thm:operatorfhcergodic} when $X$ is a Banach space with non-trivial cotype.
\begin{corollary} Let $X$ be a Banach space with non-trivial cotype, and let $T\in\mathcal L(X)$. Assume that there exists a bilateral backward orbit $(x_n)_{n\in\ZZ}$ for $T$ such that 
$\overline{\rm span}\bigl( x_n\,:\, n\in\ZZ\bigr)=X$ and the series $\sum \pm x_n$ is convergent for almost every choice of signs $\pm$. Then $T$ has countable Lebesgue spectrum after extension. 
\end{corollary}
\begin{proof} By assumption on $X$, almost sure convergence of the Rademacher series $\sum \pm x_n$ is equivalent to almost sure convergence of the Gaussian series $\sum g_n x_n$ (this follows from \cite[Corollaire 1.3 p. 67]{MP}; see also e.g. \cite[Proposition 9.14]{LT}). So the result follows immediately from Proposition~\ref{backward}.
\end{proof}

\subsection{Perfect spanning and hereditary UFHC}  The  links between properties of unimodular eigenvectors of an operator $T$ and frequent hypercyclicity of $T$ have been very much studied since~\cite{BAYGRITAMS}. The strongest available result may be the following (see \cite{newclass}, \cite{BAYMATHERGOBEST}): 
\par\smallskip
\emph{Let $X$ be a separable complex Fr\'echet space and let $T\in\mathfrak L(X)$. If the $\TT$-eigenvectors of $T$ are \textbf{perfectly spanning}, then $T$ is frequently hypercyclic, and in fact there exists a Gaussian $T$-$\,$invariant measure $\mu$ with full support such that $T$ is weakly mixing with respect to $\mu$.}  
\par\smallskip
The perfect spanning assumption means  that for any countable set $N\subset \TT$, the eigenvectors of $T$ with eigenvalues in 
$\TT\setminus N$ span a dense linear subspace of $X$; equivalently (see \cite[Proposition 6.1]{GriM}), the $\TT$-eigenvectors of $T$ are $\sigma$-spanning for some continuous probability measure $\sigma$ on $\TT$. It is plausible that under this assumption, the operator $T$ is in fact hereditarily frequently hypercyclic; but we are very far from being able to prove that. We would be already happy enough if we could weaken the assumptions of Theorem \ref{thm:ergodic} and prove that for any complex Banach space $X$, an operator $T\in\mathfrak L(X)$ is hereditarily frequently hypercyclic as soon as the $\TT$-eigenvectors of $T$ are spanning with respect to Lebesgue measure - but again, this seems out of reach for the moment. However, we do have the following result.

\begin{theorem}\label{HUFHC} Let $X$ be a complex  Fr\'echet space, and let $T\in\mathfrak L(X)$. If the $\TT$-eigenvectors of $T$ are perfectly spanning, then $T$ is hereditarily $\mathcal U$-frequently hypercyclic.
\end{theorem}

\smallskip
For the proof, we will need the following variant of \cite[Lemme 5]{Con77}.
\begin{lemma}\label{BluHanson} Let $(X, \mathcal B,\mu, T)$ be a measure-preserving dynamical system, and assume that $T$ is weakly mixing with respect to $\mu$. Let also $(n_k)_{k\geq 1}$ and $(k_i)_{i\geq 0}$ be two increasing sequence of integers. Assume that $n_{k_i}=O(k_i)$ as $i\rightarrow \infty$. Then, for any measurable set $V\subset X$, 
\[ \frac1{k_i}\sum_{k=1}^{k_i} \mathbf 1_V\circ T^{n_k}\xrightarrow{L^2} \mu(V)\quad\hbox{as $i\to\infty$}.\]
\end{lemma}
\begin{proof} The proof is similar to that of the classical Blum-Hanson Theorem \cite{BH}. We have
\begin{align*} \left\Vert \frac1{k_i}\sum_{k=1}^{k_i} \mathbf 1_V\circ T^{n_k}-\mu(V)\right\Vert_2^2&=
\frac1{k_i^2} \sum_{r,s=1}^{k_i} \Bigl( \mu\bigl( T^{-n_r}(V)\cap T^{-n_s}(V)\bigr) -\mu(V)^2\Bigr)\\
&=\frac2{k_i^2} \sum_{1\leq r<s\leq k_i} \Bigl( \mu\bigl( V\cap T^{-(n_s-n_r)}(V)\bigr) -\mu(V)^2\Bigl) +O\Bigl(\frac1{k_i}\Bigr).
\end{align*}

So it is enough to check that 
\begin{equation}\label{machin}
\sum_{1\leq r<s\leq k_i} \Bigl\vert \mu\bigl( V\cap T^{-(n_s-n_r)}(V)\bigr) -\mu(V)^2\Bigr\vert=o(k_i^2).
\end{equation}

In what follows, we put
\[ \gamma_{s,r} :=\Bigl\vert\mu\bigl( V\cap T^{-(n_s-n_r)}(V)\bigr) -\mu(V)^2\Bigr\vert.\]

Since $T$ is weakly mixing with respect to $\mu$, there is a set $D\subseteq \NN$ with ${\rm dens} (D)=1$ such that 
\begin{equation}\label{bimachin} \mu\bigl( V\cap T^{-d}(V)\bigr)-\mu(V)^2\to 0\quad\hbox{as $d\to\infty$, $d\in D$}.\end{equation}

Write 
\begin{align*}
\sum_{1\leq r<s\leq k_i} \gamma_{s,r}=\sum_{r=1}^{k_i}\sum_{\substack{r<s\leq k_i \\ n_s-n_r\in D}} \gamma_{s,r}+ \sum_{r=1}^{k_i}\sum_{\substack{r<s\leq k_i \\ n_s-n_r\not\in D}} \gamma_{s,r}=:\alpha_i+\beta_i.
\end{align*}

Using (\ref{bimachin}) and since $\gamma_{r,s}\leq 1$ for all $r,s$, it is not hard to check that \[\sum_{\substack{r<s\leq k_i \\ n_s-n_r\in D}} \gamma_{s,r}=o(k_i)\quad\hbox{ as $i\to\infty$, uniformly in $r$};\]
and it follows that $\alpha_i=o(k_i^2)$. Moreover, since ${\rm dens}(D)=1$, we see that 
\[ \beta_i\leq \sum_{r=1}^{k_i} \# \bigl\{ s\in (r,k_i] \, :\, n_s-n_r\not\in D\bigr\} \leq k_i\times  \# \Bigl( (0,n_{k_i}]\setminus D\Bigr) =k_i\times o(n_{k_i});\]
so $\beta_i=o(k_i^2)$ since we are assuming that $n_{k_i}=O(k_i)$. This proves (\ref{machin}).
\end{proof}

\begin{corollary}\label{BH2} Let $(X, \mathcal B,\mu, T)$ be a measure-preserving dynamical system, and assume that $T$ is weakly mixing with respect to $\mu$. Let also $A\subset \NN$ with $\overline{\rm dens}(A)>0$. If $V\subset X$ is a measurable set such that 
$\mu(V)>0$, then $\overline{\rm dens}\bigl(A\cap \mathcal N_T(x,V)\bigr)>0$ for $\mu$-almost every $x\in X$.
\end{corollary}
\begin{proof} Let $(n_k)_{k\geq 1}$ be the increasing enumeration of $A$. Since $\overline{\rm dens}(A)>0$, one can find an increasing sequence of integers $(k_i)_{i\geq 0}$ such that $n_{k_i}=O(k_i)$. By Lemma \ref{BluHanson}, one can find a subsequence $(k'_i)$ of $(k_i)$ such that $\frac1{k'_i}\sum_{k=1}^{k'_i} \mathbf 1_V\circ T^{n_k}\to  \mu(V)$ $\mu$-almost everywhere. In other words: for $\mu$-almost every $x\in X$,
\[ \frac1{k'_i}\, \#\bigl\{ k\in [1, k'_i]\, :\, n_k\in \mathcal N_T(x,V)\bigr\}\to \mu(V).\] Since $\#\bigl\{ k\in [1, k'_i]\, :\, n_k\in \mathcal N_T(x,V)\bigr\}=\#\bigl( [1, n_{k'_i}]\cap A\cap \mathcal N_T(x,V)\bigr)$ and $n_{k'_i}=O(k'_i)$, it follows that $\overline{\rm dens}\bigl(A\cap \mathcal N_T(x,V)\bigr)>0$, for $\mu$-almost every $x\in X$.
\end{proof}

\smallskip
\begin{proof}[Proof of Theorem \ref{HUFHC}] Assume that the $\TT$-eigenvectors of $T$ are perfectly spanning. By \cite{BAYMATHERGOBEST}, there exists a $T$-$\,$invariant Gaussian measure $\mu$ on $X$ with full support such that $T$ is weakly mixing with respect to $\mu$. Let $(A_i)_{i\in I}$ be a countable family of subsets of $\NN$ with positive upper density, and let $(V_i)_{i\in I}$ be a family of non-empty open subsets of $X$. It follows immediately from Corollary \ref{BH2} that one can find $x\in X$ (in fact, $\mu$-almost every $x\in X$ will do) such that $\overline{\rm dens}\bigl(A_i\cap \mathcal N_T(x,V_i)\bigr)>0$ for all $i\in I$.
\end{proof}

\smallskip Let us point out a consequence of Theorem \ref{HUFHC}.
\begin{corollary}\label{USD} If $X$ is a complex Banach space admitting an unconditional Schauder decomposition, then $X$ supports a hereditarily $\mathcal U$-frequently hypercyclic operator.
\end{corollary}
\begin{proof} It is shown in \cite{DFGP12} that such a space $X$ supports an operator with a perfectly spanning set of $\TT$-eigenvectors.
\end{proof}

\begin{remark}\label{HUFHCbis} The proof of Theorem \ref{HUFHC} shows the following: if $T\in\mathfrak L(X)$ and if there exists a $T$-$\,$invariant Borel measure with full support $\mu$ on $X$ such that $T$ is weakly mixing with respect to $\mu$, then $T$ is hereditarily $\mathcal U$-frequently hypercyclic. It would be interesting to know if the weak mixing assumption can be replaced by ergodicity. Incidentally, we don't know any example of an operator $T$ admitting an ergodic measure with full support but no weakly mixing measure with full support.
\end{remark}

\section{The $T_1\oplus T_2$ frequent hypercyclicity problem}

One of the most intriguing open problems regarding frequent hypercyclicity is to decide whether $T\oplus T$ 
is frequently hypercyclic whenever $T$ is frequently hypercyclic \cite{BAYGRITAMS}. Note that, by \cite{EEM}, the corresponding question for \emph{$\mathcal U$-frequent hypercyclicity} has a positive answer. A related problem is Question \ref{q0}, which asks whether $T_1\oplus T_2$ 
is frequently hypercyclic for \emph{every} frequently hypercyclic operators $T_1$ and $T_2$. This question is also open if we replace frequent hypercyclicity by $\mathcal{U}$-frequent hypercyclicity. As observed in \cite{GMM21}, $T_1\oplus T_2$ is \emph{hypercyclic} as soon as $T_1$ and $T_2$ are $\mathcal U$-frequently hypercyclic. In the opposite direction, it seems that the best known result is  \cite[Theorem 7.33]{GMM21} which deals with infinite sums:
there exists a sequence $(T_n)_{n\geq 1}$ of frequently hypercyclic operators on $\ell_p(\NN)$, $p>1$, 
such that the operator $T=\bigoplus_{n\geq 1}T_n$ acting on the $\ell_p$-sum $X=\bigoplus_{n\geq 1}\ell_p(\NN)$
is not $\mathcal U$-frequently hypercyclic.

\smallskip
As mentioned in the introduction, things are much simpler if we consider hereditarily frequently hypercyclic operators. We now give the detailed proof for the convenience of the reader.

\begin{proposition}\label{prop:sumhfhc}
 Let $\mathcal F\subset 2^\NN$ be a Furstenberg family, and let $X_1,X_2$ be two Polish topological vector spaces. Let also $T_1\in\mathfrak L(X_1)$ and $T_2\in\mathfrak L(X_2)$. If $T_1$ is $\mathcal F$-hypercyclic and $T_2$ is hereditarily $\mathcal F$-hypercyclic, then  
 $T_1\oplus T_2$ is $\mathcal F$-hypercyclic. If both $T_1$ and $T_2$ are hereditarily $\mathcal F$-hypercyclic, then $T_1\oplus T_2$ is hereditarily $\mathcal F$-hypercyclic.
\end{proposition}
\begin{proof} Assume that $T_1$ is $\mathcal F$-hypercyclic and $T_2$ is hereditarily $\mathcal F$-hypercyclic. Let $(V_i)_{i\in I}$ be a countable basis of open sets for $X_1\times X_2$. Without loss of generality, we may assume that each $V_i$ has the form $V_i=V_{i,1}\times V_{i,2}$, where $V_{i,1}, V_{i,2}$ are 
open in $X_1, X_2$. Let $x_1\in X_1$ be any $\mathcal F$-hypercyclic vector for $T_1$. Then, for each $i\in I$, the set $A_i:=\mathcal N_{T_1}(x_1, V_{i,1})$ belongs to $\mathcal F$. Since $T_2$ is hereditarily $\mathcal F$-hypercyclic, it follows that one can find a vector $x_2\in X_2$ such that $B_i:=A_i\cap\, \mathcal N_{T_2}(x_2, V_{i,2})\in\mathcal F$ for all $i\in I$. Then $x:=(x_1,x_2)$ is a frequently hypercyclic vector for $T:=T_1\oplus T_2$ since $\mathcal N_T(x, V_i)\supset B_i$ for all $i\in I$. 

The proof of the second part of the proposition is essentially the same.
\end{proof}

\medskip
Using weighted backward shifts on $c_0(\mathbb Z_+)$, we now find a counterexample to the $T_1\oplus T_2$ frequent hypercyclicity problem, and thus answer Question \ref{q0} in the negative. This counterexample also solves the $T_1\oplus T_2$ $\mathcal{U}$-frequent hypercyclicity problem.

\begin{theorem}\label{thm:sumfhc}
 There exist two frequently hypercyclic weighted shifts $B_w, B_{w'}$ on $c_0(\mathbb Z_+)$ such that $B_w\oplus B_{w'}$ is not $\mathcal U$-frequently hypercyclic.
\end{theorem}

From this theorem and Proposition \ref{prop:sumhfhc}, we immediately deduce the following result, which is of course to be compared with Corollary \ref{shiftlp}. 

\begin{corollary} \label{corc0}
 There exist weighted shifts  on  $c_0(\mathbb Z_+)$ which are frequently hypercyclic but not hereditarily frequently hypercyclic.
\end{corollary}

Let us also point out another consequence of Theorem \ref{thm:sumfhc} and Remark \ref{HUFHCbis}. 
\begin{corollary} There exist frequently hypercyclic weighted shifts on $c_0(\ZZ_+)$ which admit no weakly mixing invariant measure with full support, and hence no ergodic invariant Gaussian measure with full support.
\end{corollary}
This is not really a new result: the Gaussian part has been known since \cite{BAYGRILONDON} (with arguably a more complicated example than the one 
we are about to present here); and it was proved in \cite{GriM} that there exist frequently hypercyclic \emph{bilateral} weighted shifts on $c_0(\ZZ)$ which admit no ergodic invariant measure with full support.

\smallskip
In the proof of Theorem \ref{thm:sumfhc}, we shall use the following lemma, which gives a simple characterization of frequent hypercyclicity for weighted shifts  on $c_0(\mathbb Z_+)$ whose weight sequence is bounded below (see \cite{BAYRUZSA} or \cite[Corollary 34]{BoGre18}).

\begin{lemma}\label{lem:fhcws}
 Let $w=(w_n)_{n\geq 1}$ be a bounded sequence of positive real numbers and assume that $\inf_{n\geq 1} w_n>0$. 
 Then the associated weighted shift $B_w$ is frequently hypercyclic on $c_0(\ZZ_+)$ if and only if there exist a sequence $(M(p))_{p\geq 1}$ of positive real numbers tending to infinity 
 and a sequence $(E_p)_{p\geq 1}$ of disjoint subsets of $\NN$ with positive lower density such that 
 
 \smallskip
 \begin{itemize}
  \item[\rm (a)] $\displaystyle\lim_{n\to\infty,\ n\in E_p}w_1\cdots w_n=\infty$ for all $p\geq 1$;
  \item[\rm (b)] for all $p,q\geq 1,$ for all $m\in E_p$ and $n\in E_q$ with $m>n,$
  $$w_1\cdots w_{m-n}\geq \max(M(p),M(q)).$$
 \end{itemize}
\end{lemma}

 \smallskip 
 We will also need the following elementary lemma, which is almost the same as \cite[Lemma 6.1]{BAYRUZSA}. For $\veps>0$, $a>1$ and $u\in\NN,$ we let
\[ I_u^{a,\veps}:=[(1-\veps)a^u,(1+\veps)a^u].\]
\begin{lemma}\label{lem:sets}
 There exist $\veps>0$ and $a>1$ such that, for any integers $u>v\geq 1,$
 $$I_u^{a,4\veps}\cap I_v^{a,4\veps}=\emptyset,\quad I_u^{a,2\veps}-I_v^{a,2\veps}\subset I_u^{a,4\veps}\quad{\rm and}\;\; I_v^{a,\veps}+[-v,v]\subset I_v^{a,2\veps}.$$
\end{lemma}
\begin{proof}
 Provided that $\veps\in (0,1/4),$ the first condition is equivalent to saying that
 $$(1+4\veps)a^u<(1-4\veps)a^{u+1}\qquad\hbox{for all $u\geq 1$},$$
 \textit{i.e.}
 \[ \frac{1+4\veps}{(1-4\veps)a}<1.\]
 
 The second one is satisfied as soon as
 $$(1-2\veps)a^u-(1+2\veps)a^{u-1}\geq (1-4\veps)a^u\qquad\hbox{for all $u\geq 2,$}$$
 which is equivalent to
 \[ \frac{2\veps a}{1+2\veps}\geq 1.\]
 
 The last condition is satisfied if $(1-\veps)a^v-v\geq (1-2\veps)a^v$ 
  for all $v\geq 1,$  in other words
 $$\veps a^v\geq v\qquad\hbox{for all $v\geq 1$}.$$
  
  Therefore one can choose e.g. $\veps:=1/8$ and then take $a$ sufficiently large.
\end{proof}

\begin{proof}[Proof of Theorem \ref{thm:sumfhc}] Our construction is inspired by that of \cite[Section 6]{BAYRUZSA}.  In what follows, we fix once and for all $\veps>0$ and $a>1$ satisfying the conclusions of Lemma \ref{lem:sets}.
 
 \smallskip
  For $k\geq 1,$ let 
 \[ A_k:=2^{k-1}\mathbb N\backslash 2^k\NN.\]
Note that each $A_k$ is a syndetic set, \textit{i.e.} it has bounded gaps, and the sets $A_k$ are pairwise disjoint.  
Moreover, since $2^{k-1}\geq k$, we have 
 $I_v^{a,\veps}+[-k,k]\subset I_v^{a,2\veps}$ for each $k\geq 1$ and all $v\in A_k$. 
 
 We also fix an increasing sequence of positive integers $(b_p)_{p\geq 1}$
 such that 
 $$\lim_{p\to\infty}\udens \left[\bigcup_{q\geq p} \bigl( b_q\NN+[-q,q]\bigr)\right]=0.$$
 
Finally, we set for $p\geq 1,$
 \begin{align*}
  E_p:=&\bigcup_{u\in A_{2p}}\left(I_u^{a,\veps}\cap b_{p}\NN\right),\\
  F_p:=&\bigcup_{u\in A_{2p+1}}\left(I_u^{a,\veps}\cap b_{p}\NN\right).\\ 
 \end{align*}
 
\smallskip 
 
We note that since $A_{2p}$ and $A_{2p+1}$ are syndetic, we have
\[\ldens(E_p)>0\quad{\rm and}\quad \ldens(F_p)>0\qquad\hbox{for all $p\geq 1$}.\]

\smallskip\noindent
Indeed, for all $p\geq 1,$ there exists $\delta_p>0$ such that, for all $u$ sufficiently large, 
$$\#\left(I_{u}^{a,\veps}\cap b_p\NN\right)\geq \delta_p a^u.$$
Let $R_p$ be such that if $u$ and $v$ are two consecutive elements of $A_{2p}$ then $v-u\leq R_p.$
If now $n$ is very large and if we consider $u$ and $v$ two consecutive elements of $A_{2p}$ such that
$$(1+\veps)a^u<n \leq (1+\veps)a^v,$$
then we see that
$$\frac{\#\big(E_p\cap [1,n]\big)}{n}\geq \frac{\#\big(I_u^{a,\veps}\cap b_p\NN\big)}{(1+\veps)a^v}\geq\frac{\delta_p}{(1+\veps)a^{R_p}}\cdot$$
Hence, $\ldens(E_p)>0$. A similar argument shows that $\ldens(F_p)>0$.

\smallskip
We now construct our weight sequences $w$ and $w'$. 

For $p\geq 1,$ we first define a sequence $(w_n^p)_{n\geq 1}\subset\{1/2,1, 2\}$ such that, for all $n\geq 1,$
$$w_1^p\cdots w_n^p=\left\{
\begin{array}{lll}
 1&\textrm{if }n\notin I_u^{a,2\veps},&u\in A_{2p}\\
 2^u&\textrm{if }n\in I_u^{a,\veps},&u\in A_{2p}.
\end{array}\right.$$
This is possible since $I_u^{a,\veps}+[-u,u]\subset I_u^{a,2\veps}$. These sequences will be used for handling condition (a) in Lemma 
\ref{lem:fhcws}.

For $p\geq 1,$ we also define a sequence $(\omega_n^p)_{n\geq 1}\subset \{1/2,1, 2\}$ such that, for all $n\geq 1,$
$$\omega_1^p\cdots \omega_n^p=\left\{
\begin{array}{ll}
 1&\textrm{if }n\notin b_{p}\NN+[-p,p]\\
 2^p&\textrm{if }n\in b_{p}\NN.
\end{array}\right.$$
These sequences will help us to verify condition (b) in Lemma \ref{lem:fhcws} for $p=q$.

Finally, for $u>v\geq 1$ with $u\in A_{2p}$ and $v\in A_{2q}$ for some $p,q\geq 1,$ we define a sequence $(w_n^{u,v})_{n\geq 1}\subset\{1/2,1, 2\}$
such that, for all $n\geq 1,$
$$w_1^{u,v}\cdots w_n^{u,v}=\left\{
\begin{array}{ll}
 1&\textrm{if }n\notin I_u^{a,4\veps}\\
 \max(2^p,2^q)&\textrm{if }n\in I_u^{a,\veps}-I_v^{a,\veps}.
\end{array}\right.$$
This is possible since 
\begin{align*}
 I_u^{a,\veps}-I_v^{a,\veps}+[-\max(p,q),\max(p,q)]&\subset (I_u^{a,\veps}+[-p,p])-(I_v^{a,\veps}+[-q,q])\\
 &\subset I_u^{a,2\veps}-I_v^{a,2\veps}\\
& \subset I_u^{a,4\veps}.
\end{align*}
These sequences will be needed in order to check condition (b) in Lemma \ref{lem:fhcws} for $p\neq q$.

We finally define the weight sequence $w=(w_n)_{n\geq 1}$ as follows: for all $n\geq 1$, we require that
$$w_1\cdots w_n=\max_{p,u>v}(w_1^p\cdots w_n^p,\omega_1^p\cdots \omega_n^p,w_1^{u,v}\cdots w_n^{u,v}).$$
It is not difficult to check that $w_n\in [1/2, 2]$ for all $n\geq 1$. Indeed, assume for instance that $w_1\cdots w_n=w_1^p\cdots w_n^p.$ Then 
$w_1\cdots w_{n-1}\geq w_1^p\cdots w_{n-1}^p$ and 
$$w_n\leq w_n^p\leq 2.$$
The same argument works for the other cases; for the lower bound, assume for instance that $w_1\cdots w_{n-1}=w_1^p\cdots w_{n-1}^p.$ Then $w_1\cdots w_n\geq w_1^p\cdots w_n^p$ so that 
$$w_n\geq w_n^p\geq \frac 12.$$
We define in a similar way a weight sequence $w'=(w'_n)_{n\geq 1}\subset[1/2,2]$,
replacing everywhere $A_{2p}$ by $A_{2p+1}$ and $A_{2q}$ by $A_{2q+1}$. 

\smallskip Let us first show that $w$ and $w'$ satisfy the conditions of Lemma \ref{lem:fhcws}, so that $B_w$ and $B_{w'}$ are frequently hypercyclic. It is clearly enough to do that for $w$.

\begin{enumerate}[(a)]
 \item If $n\in E_p$, there is a unique $u=u(n)\in A_{2p}$ such that $n\in I_u^{a,\veps}$. Then $w_1\cdots w_n\geq w_1^p\cdots w_n^p\geq 2^{u(n)}$, which shows that $w_1\cdots w_n\to \infty$ as $n\to\infty$, $n\in E_p$.
 \item Let $p,q\geq 1,$, and let us fix $m\in E_p$ and $n\in E_q$ with $m>n$. If $p=q$ then $m-n\in b_p \NN$, so that  
 $w_1\cdots w_{m-n}\geq \omega_1^p\cdots\omega_{m-n}^p\geq 2^p.$ 
If $p\neq q$, there exist $u>v\geq 1$ such that $m\in I_{u}^{a,\veps}$ and $n\in I_v^{a,\veps}$, and then 
 $w_1\cdots w_{m-n}\geq w_1^{u,v}\cdots w_{m-n}^{u,v}\geq\max(2^p,2^q).$
\end{enumerate}

\smallskip Now, let us show that  $B_w\oplus B_{w'}$ is not $\mathcal U$-frequently hypercylic. Denote by $(e_j)_{j\ge 0}$ the canonical basis of $c_0(\mathbb Z_+)$.  We show that for any vector $x\in c_0(\mathbb Z_+)\oplus c_0(\mathbb Z_+)$, the set 
$E_x:=\bigl\{n\in\NN:\ \|(B_w\oplus B_{w'})^nx-(e_0,e_0)\|<1/2\bigr\}$ has upper density equal to $0$.

Towards a contradiction, assume that $\overline{\rm dens}(E_x)>0$ for some vector $x$. It is easy to check that 
$$\lim_{n\to\infty,\ n\in E_x}w_1\cdots w_n =\lim_{n\to\infty,\ n\in E_x}w'_1\cdots w'_n=\infty.$$
It follows that if we set
$$G_p:=\{n\in\NN:\ w_1\cdots w_n\geq 2^p\textrm{ and }w'_1\cdots w_n'\geq 2^p\},$$
then $E_x\setminus G_p$ is finite and hence  $ \udens(G_p)\geq \udens(E_x)>0$ for all $p\geq 1$. 

Now the construction 
of the weight sequence $w$  yields that if $n\in G_p$, then either $n\in b_q\NN+[-q,q]$ for some $q\geq p$, or 
$n\in I_u^{a,4\veps}$ for some $u\in \bigcup_{q\geq 1}A_{2q}$. Similarly, by construction of the sequence $w'$ we also know that if $n\in G_p$, then either $n\in b_q\NN+[-q,q]$ for some $q\geq p$, or 
$n\in I_v^{a,4\veps}$ for some $v\in \bigcup_{q\geq 1}A_{2q+1}$. By disjointness of the sets $I_u^{a,4\veps}$ and $I_v^{4,\veps}$ for $u\neq v,$ it follows that 
$$G_p\subset\bigcup_{q\geq p}\bigl( b_q\NN+[-q,q]\bigr).$$

By our choice of the sequence $(b_p),$ we get a contradiction with $\udens(G_p)\geq \udens(E_x)>0$.
\end{proof}
\begin{remark}
The weighted shift $B_w$ cannot serve as an counterexample to the $T\oplus T$ frequent hypercyclicity problem. Indeed, it can be shown (see \cite[Theorem 18]{karl}) that any weighted shift on $c_0(\ZZ_+)$
 satisfying the assumptions of Lemma \ref{lem:fhcws} is such that any finite direct sum $B_w\oplus\cdots\oplus B_w$
 is itself frequently hypercyclic.
\end{remark}

\section{FHC operators on $\ell_p({\mathbb Z}_+)$ which are not hereditarily FHC}
\subsection{The result} In this section, we use the machinery developed in \cite{GMM21}, following the construction in \cite{Me1} of chaotic operators which are not frequently hypercyclic, to produce an operator on $\ell_p(\ZZ_+)$, $1\leq p<\infty$ which is frequently hypercyclic but not hereditarily frequently hypercyclic. We will in fact obtain a formally  stronger result.

\begin{definition}\label{along}
Let $A\subset\NN$ be a set with $\underline{\text{dens}}(A)>0$. We say that an operator $T\in\mathfrak L(X)$ is \textbf{frequently hypercyclic along $A$} if the sequence $(T^n)_{n\in A}$ is frequently hypercyclic: there exists $x\in X$ such that $\underline{\rm dens}\,\bigl( {A\cap \mathcal N_T(x,V)}\bigr)>0$ for every non-empty open set $V\subset X$.
\end{definition}

Obviously, if an operator is hereditarily frequently hypercyclic, then it is frequently hypercyclic along any set $A\subset\NN$ with positive lower density. Our aim is to prove the following theorem.

\begin{theorem}\label{Ctypeex}
Let $1\le p<\infty$. There exist an operator $T$ on $\ell_p(\mathbb{Z}_+)$ and a set $A\subset\NN$ with $\underline{\text{\emph{dens}}}(A)>0$ such that $T$ is frequently hypercyclic and chaotic,  but not frequently hypercyclic along $A$ {\rm (}and thus not hereditarily frequently hypercyclic{\rm )}.
\end{theorem}

\medskip With the terminology of \cite{GMM21}, the operator we are looking for will be a {\cput\ operator}.  So we will need to recall the definition of  \cput\  operators, and more generally of {C-type} and \cpt\   operators. But before that, we will prove a general result allowing to check in a simple way that an operator is not frequently hypercyclic along some set with positive lower density.

\subsection{How not to be hereditarily FHC} The next  theorem gives simple conditions ensuring that an operator is not hereditarily frequently hypercyclic

\begin{theorem}\label{nothfhc}
Let $X$ be a Banach space admitting a Schauder basis $(e_k)_{k\ge 0}$, and let $T\in \mathfrak L(X)$. Denoting by $\pi_K$, $K\geq 1$ the canonical projection onto ${\rm span}(e_k\,:\, 0\leq k\leq K-1)$, assume that there exist  increasing sequences of integers $(K_n)_{n\ge 1}$ and $(J_n)_{n\ge 1}$ such that for every $n$:
\begin{enumerate}
\item[\rm (a)] $T^{J_n} \pi_{K_n}=\pi_{K_n}$;
\item[\rm (b)] $\|\pi_{K_n}T^j (I-\pi_{K_n})x\|\le \|(I-\pi_{K_n})x\|\quad $
for all $x\in X$ and $0\le j\le (n+1)^{2^{J_n}}J_n$. 
\end{enumerate}
Then there exists a set $A\subset\NN$ with positive lower density such that $T$ is not hereditarily frequently hypercyclic along $A$. 
\end{theorem}
\begin{proof} Extracting subsequences of $(K_n)$ and $(J_n)$ if necessary, we may assume without loss of generality that $J_{n+1}\ge 2(n+1)^{2^{J_n}}J_n$ for all $n$. 

\smallskip Let us denote by $\mathcal{F}_n\subset 2^{\mathbb{N}}$ the family of all finite sets $F\subset [0,J_n)$ such that $\#F\ge J_n/2$
. Let $C_n:=\#\mathcal{F}_n$
, and let $(F_{n,j})_{0\le j< C_n}$ an enumeration of $\mathcal{F}_n$. We set $M_n:=(n+1)^{C_{n}}J_n$ and we remark that $M_n\le (n+1)^{2^{J_n}}J_n\le J_{n+1}/2$.

\smallskip 
We now construct the set $A$ by induction. To start, let
\[A_1:=[0,J_1)\cup \bigcup_{0\le j< C_1}\bigcup_{0\le l<2^{j+1}-2^j}\left((2^j+l)J_1+F_{1,j}\right).\]
Given $k\ge 2$, if $A_1,\dots, A_{k-1}$ have been defined already, we define $A_k$ by setting
\begin{align*}
A_{k}:=[M_{k-1},J_{k})\; \cup \; \bigcup_{0\le j<C_{k}}\bigcup_{0\le l<(k+1)^{j+1}-(k+1)^j}\left((k+1)^j+l)J_{k}+F_{k,j}\right).
\end{align*}
Observe that the sets $(k+1)^j+l)J_{k}+F_{k,j}$ involved in this definition are pairwise disjoint, since they are contained in successive intervals. More precisely,
$$A_k\cap [sJ_{k}, (s+1)J_{k})=
sJ_{k}+F_{k,j}$$ for every $0\le j<C_k$ and every $(k+1)^{j}\le s<(k+1)^{j+1}$,
and  $\max (A_k)< (k+1)^{C_k}J_k= M_k$. Hence $A_k\subseteq[M_{k-1}, M_k)$. Finally, we let
\[ A:=\bigcup_{k\ge 1} A_k.\]

\begin{claim}\label{1/4} We have $\frac{\#(A\cap [0,N])}{N+1}\ge 1/4$ for all $N\geq 0$. In particular, \emph{$\ldens(A)>0$.}
\end{claim}
\begin{proof}[Proof of Claim \ref{1/4}] We will check by induction on $k\ge 1$ that
\begin{equation}\label{eqc}
\frac{\#( A\cap [0,N])}{N+1}\ge 1/4
\quad\textrm{ for all } N<M_k.
\end{equation}

If $N<M_1$, there exists $0\le s < 2^{C_1}$ such that $s J_1\le N<(s+1)J_1$. Then $\frac{\#(A\cap [0,N])}{N+1}\ge 1$
if $s=0$ and 
\[\frac{\#(A\cap [0,N])}{N+1}\ge \frac{\#(A_1\cap [0,sJ_1])}{(s+1)J_1}\ge \frac{s}{2(s+1)}\ge \frac{1}{4}\qquad \text{if $s\ge 1$}\]
because the sets involved in the definition of $A_1$ are pairwise disjoint and $\#F_{1,j}\ge J_1/2$
 for every $j<C_1$. 
 \par\smallskip
 Assume that the inequality (\ref{eqc}) has been proved for $k-1$. In order to get the result for $k$, it is enough to check that $$\frac{\#(A\cap [0,N])}{N+1}\ge \frac{1}{4} \quad\textrm{ for every } M_{k-1}\le N<M_{k}.$$
 If $M_{k-1}\le N< J_{k}$ then, since  $[ M_{k-1}, N]\subset [M_{k-1},J_{k})\subset A_{k}$, we get by the induction assumption that $$\frac{\#(A\cap [0,N])}{N+1}\ge \frac{\#(A\cap [0,M_{k-1}))}{M_{k-1}}\geq \frac{1}{4}\cdot$$
 Moreover, we even have $\frac{\#(A\cap [0,J_{k}))}{J_{k}}\ge \frac{1}{2}$
 since $[M_{k-1},J_{k})\subset A_{k}$ and $M_{k-1}\le J_{k}/2$. On the other hand, if $J_{k}\le N<M_{k}$, there exists $1\le s < (k+1)^{C_{k}}$ such that $s J_{k}\le N<(s+1)J_{k}$ and we obtain in this case that
\[\frac{\#(A\cap [0,N])}{N+1}\ge \frac{\#(A\cap [0,J_{k}))}{(s+1)J_{k}}+\frac{\#(A_{k}\cap [J_{k},sJ_{k}])}{(s+1)J_{k}}\ge \frac{1}{2(s+1)}+\frac{s-1}{2(s+1)}\ge \frac{1}{4}.\]
This proves Claim \ref{1/4}.
\end{proof}
 
 Let us now get back to the proof of Theorem \ref{nothfhc}. Our aim is to show that under assumptions (a) and (b),
$T$ is not hereditarily frequently hypercyclic along the set $A$ that we just constructed. To this end, we consider for any $c>0$, the open sets \[ U_c:=\bigl\{y\in X\,:\, |\langle e^*_0, y\rangle|<c\bigr\}\qquad{\rm and}\qquad V_c=\bigl\{y\in X\,:\, |\langle e^*_0, y\rangle|>c\bigr\},\]
and we show that for every $x\in X$, either $\mathcal N_T(x,U_{1/2})\cap A$ or $\mathcal N_T(x,V_{3/2})\cap A$ has a lower density equal to $0$. Let $x\in X$. Since $\vert \langle e_0^*, u\rangle\vert \leq C\, \Vert \pi_{K_n} u\Vert$ for some absolute constant $C>0$, it follows from assumption (b) that for $n$ sufficiently large, we have
\begin{align*}
\mathcal N_T(x,U_{1/2})\cap [0,(n+1)^{2^{J_n}}J_n)&\subset \mathcal N_T(\pi_{K_n}x,U_{3/4})\cap [0,(n+1)^{2^{J_n}}J_n)\\
&\subset \mathcal N_T(x,U_{1})\cap [0,(n+1)^{2^{J_n}}J_n)
\end{align*}
and 
\begin{align*}
\mathcal N_T(x,V_{3/2})\cap [0,(n+1)^{2^{J_n}}J_n)&\subset \mathcal N_T(\pi_{K_n}x,V_{4/3})\cap [0,(n+1)^{2^{J_n}}J_n)\\
&\subset \mathcal N_T(x,V_{1})\cap [0,(n+1)^{2^{J_n}}J_n).
\end{align*}
Moreover, since $\mathcal N_T(x,U_1)\cap\, \mathcal N_T(x,V_1)=\emptyset$, we have, for any $n\geq 1$, that  $${\rm either}\quad \#\bigl(\mathcal N_T(x,U_{1})\cap [0,J_n)\bigr)\le J_n/2 \quad\textrm{ or }\quad \#\bigl(\mathcal N_T(x,V_{1})\cap [0,J_n)\bigr)\le J_n/2.$$
Hence, either $\#\bigl(\mathcal N_T(x,U_{1})\cap [0,J_n)\bigr)\le J_n/2$
for infinitely many $n$'s or $\#\bigl( \mathcal N_T(x,V_{1})\cap [0,J_n)\bigr)\le J_n/2$
for infinitely many $n$'s. Without loss of generality, we assume that $\#\bigl( \mathcal N_T(x,U_{1})\cap [0,J_n)\bigr)\le J_n/2$
for infinitely many $n$'s (the other case being similar). Hence, there exists an increasing sequence $(n_k)_{k\ge 1}$ of integers  and a sequence $(j_k)_{k\ge 1}$ of integers such that
\[\mathcal N_T(x,U_{1})\cap [0,J_{n_k})\cap F_{n_k,j_k}=\emptyset \quad\textrm{ for every } k\ge 1.\]
Now, since $T^{J_{n_k}} \pi_{K_{n_k}}=\pi_{K_{n_k}}$ by assumption (a), we have for every $k\ge 1$,
\begin{align*}
\mathcal N_T(x,U_{1/2})\cap [&0,(n_k+1)^{2^{J_{n_k}}}J_{n_k})\\
&\qquad\subset \;\mathcal N_T(\pi_{K_{n_k}}x,U_{3/4})\cap [0,(n_k+1)^{2^{J_{n_k}}}J_{n_k})\\
&\qquad=\bigcup_{0\le l<(n_k+1)^{2^{J_{n_k}}}} \left(lJ_{n_k}+ \mathcal N_T(\pi_{K_{n_k}}x,U_{3/4})\cap[0,J_{n_k}) \right)\\
&\qquad \subset \bigcup_{0\le l<(n_k+1)^{2^{J_{n_k}}}} \Bigl(lJ_{n_k}+ \mathcal N_T(x,U_{1})\cap[0,J_{n_k}) \Bigr).
\end{align*}
Intersecting with $A$, and observing that $(n_k+1)^{j_k+1}\le (n_k+1)^{2^{J_{n_k}}}$, we get that
\begin{align*}
\bigl(\mathcal N_T(x,U_{1/2})\cap A\bigr)\cap [0,(&n_k+1)^{j_k+1}J_{n_{k}})\\
&\qquad\subset
A\cap \bigcup_{0\le s<(n_k+1)^{j_k+1}} \Bigl(sJ_{n_k}+\mathcal N_T(x,U_{1})\cap[0,J_{n_k})\Bigr).
\end{align*}
Now, by definition of $A$,
we have  $$A\cap \bigcup_{(n_k+1)^{j_k}\le s<(n_k+1)^{j_k+1}}[sJ_{n_k}, (s+1)J_{n_k})=
\bigcup_{(n_k+1)^{j_k}\le s<(n_k+1)^{j_k+1}}\left(sJ_{n_k}+F_{n_k,j_k}\right).$$
Since $\mathcal N_T(x,U_{1})\cap [0,J_{n_k})\cap F_{n_k,j_k}=\emptyset$,
it follows that
\begin{equation*}
\bigl(\mathcal N_T(x,U_{1/2})\cap A\bigr)\cap [(n_k+1)^{j_k}J_{n_k},(n_k+1)^{j_k+1}J_{n_{k}})
=\emptyset,
\end{equation*}
so that
\begin{equation*}
\bigl(\mathcal N_T(x,U_{1/2})\cap A\bigr)\cap [0,(n_k+1)^{j_k+1}J_{n_{k}})
\subset [0,(n_k+1)^{j_k}J_{n_k}).
\end{equation*}
So we see that
\[ \frac{\#\bigl ((\mathcal N_T(x,U_{1/2})\cap A)\cap [0,(n_k+1)^{j_k+1}J_{n_{k}})\bigr)}{(n_k+1)^{j_k+1}J_{n_{k}}}\le  \frac{(n_k+1)^{j_k}J_{n_k}}{(n_k+1)^{j_k+1}J_{n_{k}}}\cdot\] The right hand side of this inequality tends to $0$ as $n$ tends to infinity,
and this shows that $\ldens\bigl(\mathcal N_T(x,U_{1/2})\cap A\bigr)=0$.
\end{proof}
\begin{remark} Assumption (b) in Theorem \ref{nothfhc} can be weakened: it is enough to assume that there exists a non-zero linear functional $x^*\in X^*$ such that $\vert \langle x^*, T^j (I-\pi_{K_n})x\rangle\vert\le \|(I-\pi_{K_n})x\|$
for all $x\in X$ and $j\le (n+1)^{2^{J_n}}J_n$. This is apparent from the above proof.
\end{remark}

\subsection{C-type operators} We recall here very succintly some basic facts concerning C-type operators, and we refer the reader to \cite[Sections 6 and 7]{GMM21} for more on this class of operators. In what follows, we denote by $(e_k)_{k\geq 0}$ the canonical basis of $\ell_p(\ZZ_+)$, $1\le p<\infty$.

\smallskip Let us consider four 
``parameters'' $v$, $w$ $\varphi$ and $b$, where

\smallskip
\begin{enumerate}
 \item[-] $v=(v_{n})_{\gn}$ is a sequence of non-zero complex numbers such 
that $\sum_{\gn}|v_{n}|<\infty $;
\item[-] $w=(w_{j})_{j\geq 1}$ is a sequence of complex numbers
such that $0<\inf_{k\ge 1} \vert w_k\vert\leq \sup_{k\ge 1}\vert w_k\vert<\infty$;
\item[-] $\varphi $ is a map from $\ZZ_+$ into itself, such that $\varphi 
(0)=0$, $\varphi (n)<n$ for every $\gn$, and the set 
$\varphi ^{-1}(l)=\{n\ge 0\,;\,\varphi (n)=l\}$ is infinite for every 
$l\ge 0$;
\item[-] $b=(b_{n})_{n\ge 0}$ is a strictly increasing sequence of positive 
integers such that $b_{0}=0$ and $b_{n+1}-b_{n}$ is a multiple of 
$2(b_{\varphi (n)+1}-b_{\varphi (n)})$ for every $n\ge 1$.
\end{enumerate}

\smallskip 
If $w$ and $b$ are such that \[ \inf_{n\geq 0} \prod_{b_n<j<b_{n+1}}\, \vert w_j\vert>0,\] then, by  \cite[Lemma 6.2]{GMM21},  there is a unique bounded operator $\tvw$ on $\ell_p(\ZZ_+)$ such that

\[
\tvw\ e_k=
\begin{cases}
 w_{k+1}\,e_{k+1} & \textrm{if}\ k\in [b_{n},b_{n+1}-1),\; n\geq 0\\
v_{n}\,e_{b_{\varphi(n)}}-\Bigl(\,\,\ds\prod_{j=b_{n}+1}^{b_{n+1}-1}
w_{j}\Bigr)^{ -1 } e_ {
b_{n}} & \textrm{if}\ k=b_{n+1}-1,\ \gn\\
 -\Bigl(\!\!\ds\prod_{j=b_0+1}^{b_{1}-1}w_j\Bigr)^{-1}e_0& \textrm{if}\ 
k=b_1-1.
\end{cases}
\]

\medskip Any such operator $\tvw$ is called a \emph{C-type operator}. A notable fact to be pointed out immediately is that C-type operators have lots of periodic points; indeed, we have the following fact, which is \cite[Lemma 6.4]{GMM21}.
\begin{fact}\label{period} If $T=\tvw$ is a C-type operator, then 
\[T^{2(b_{n+1}-b_n)}e_k=e_k\qquad\hbox{if $k\in [b_n,b_{n+1}), \; n\geq 0$.}\]
\end{fact}

It follows that every finitely supported vector is periodic for $\tvw$; in particular, a C-type operator is chaotic as soon as it is hypercyclic.

\smallskip A \emph{\cpt\ operator} is a C-type operator for which the parameters  satisfy the following additional conditions: for every $k\geq 1$, 

\smallskip
\begin{enumerate}
 \item[-] $\varphi$ is increasing on each interval $[2^{k-1},2^{k})$ with  $\varphi \bigl([2^{k-1},2^{k})\bigr)=[0,2^{k-1})$, \textit{i.e.} \[\varphi (n)=n-2^{k-1}\quad\hbox{for every $n\in[2^{k-1},2^{k})$};\]
 \item[-] the blocks $[b_{n},b_{n+1})$, $n\in[2^{k-1},2^{k})$ 
all have the same size, which we denote by $\Delta ^{(k)}$: 
\[b_{n+1}-b_{n}=
\Delta ^{(k)}\qquad\hbox{for every $n\in[2^{k-1},2^{k})$};
\]
\item[-] the sequence $v$ is constant on the interval $[2^{k-1},
2^{k})$: there exists $v^{(k)}$ such that 
\[ v_{n}=v^{(k)}\qquad\hbox{for every 
$n\in[2^{k-1},2^{k})$};
\]
\item[-] the sequences of weights $(w_{b_{n}+i})_{1\le i
<\Delta ^{(k)}}$ are independent of $n\in[2^{k-1},2^{k})$: there exists a 
sequence $(w_{i}^{(k)})_{1\le i<\Delta ^{(k)}}$ such that 
\[
w_{b_{n}+i}=w_{i}^{(k)} \quad\hbox{for every $n\in[2^{k-1},2^{k})$ and $1\le i<\Delta ^{(k)}$.}
\]
\end{enumerate}

\smallskip Finally, a \emph{\cput\ operator} is a \cpt\ operator whose parameters are such that for all $k\geq 1$, 

\[
v^{(k)}=2^{-\tau ^{(k)}}\quad\textrm{and}\quad w_{i}^{(k)}=
\begin{cases}
2&\textrm{if}\ \ 1\le i\le \delta ^{(k)}\\
1&\textrm{if}\ \ \delta ^{(k)}<i<\Delta ^{(k)}
\end{cases}
\]
where $(\tau ^{(k)})_{k\ge 1}$ and $(\delta ^{(k)})_{k\ge 1}$ are two 
 increasing sequences of integers with $\delta ^{(k)}<\Delta 
^{(k)}$ for each $k\ge 1$.

\smallskip
These operators have been studied in detail in \cite[Section 7]{GMM21}. In particular, we have the following crucial fact (\cite[Theorem 7.1]{GMM21}).

\begin{fact}\label{cputfhc0} A \cput\ operator $\tvw$ is frequently hypercyclic as soon as   
\begin{equation}\label{cputfhc}  \limsup\limits_{k\to\infty }\,
  \dfrac{\delta ^{(k)} -\tau ^{(k)}}{\Delta ^{(k)}}>0.
  \end{equation}
\end{fact} 

\subsection{Proof of Theorem \ref{Ctypeex}}
Let $T=\tvw$ be an operator of C$_{+,1}$-type on $\ell_{p}(\ZZ_+)$, so that $v$ and $w$ are given by
\[
v^{(k)}=2^{-\tau ^{(k)}}\quad\textrm{and}\quad w_{i}^{(k)}=
\begin{cases}
2&\textrm{if}\ \ 1\le i\le \delta ^{(k)}\\
1&\textrm{if}\ \ \delta ^{(k)}<i<\Delta ^{(k)}.
\end{cases}
\]

We assume that $\Delta^{(k)}\in 8\NN$ for all $k\ge 1$, and we choose
\[\delta^{(k)}:=\frac{1}{4}\Delta^{(k)} \qquad \text{and} \qquad \tau^{(k)}:=\frac{1}{8}\Delta^{(k)}.\]

\noindent So the only ``free'' parameter is now the sequence $\bigl(\Delta^{(k)}\bigr)_{k\geq 1}$.

\smallskip

By Fact \ref{cputfhc0}, the operator $T$ is frequently hypercyclic (and hence chaotic since it is a C-type operator). So we just have to show that if the sequence $(\Delta^{(k)})$ is suitably chosen, then $T$ satisfies the assumptions of Theorem~\ref{nothfhc}. We will in fact show that this holds as soon as the sequence $(\Delta^{(k)})$ grows sufficiently rapidly. Let us set
\[ K_n:=b_{2^{n}}\qquad{\rm and}\qquad J_n:=2\Delta^{(n)}\quad\textrm{ for every }n\ge 1.\]

\smallskip With this choice of the sequences $(K_n)$ and $(J_n)$, condition (a) in Theorem \ref{nothfhc} is satisfied by Fact \ref{period}. So the only thing to check is condition (b).

\smallskip
Let $\gamma_k:=2^{\,\delta ^{(k-1)}-\tau^{(k)}}\bigl(\Delta ^{(k)}\bigr)^{1-\frac1{p}}$. If $(\Delta ^{(k)})$ grows sufficiently rapidly, then the sequence $(\gamma_k)$ is decreasing and 
\[2^n \sum_{k\ge n+1}2^{k-1}\gamma_k\le 1\qquad\hbox{for all $n\geq 0$}.\]

Let us also define a sequence $(\beta_l)_{l\geq 1}$ as follows:
\[\beta _{l}:=4\,\gamma_k\qquad\hbox{if $l\in [2^{k-1},2^{k})$.}\]

Finally, for any $l\geq 1$, let $P_l$ be the projection of $\ell_p({\mathbb Z}_+)$ defined by
\[ P_lx=\sum_{k=b_l}^{b_{l+1}-1}x_ke_k \quad\textrm{ for every } x\in\ell_p({\mathbb Z}_+). \]

As in the proof of \cite[Theorem 7.2]{GMM21} one can show that the following estimate holds for every $k\ge 0$, every $l\in [2^{k-1},2^k[$, every $0\le m<l$ and every $0\le j\le \Delta^{(k)}-\delta^{(k)}=\frac{3}{4}\Delta^{(k)}$:
\[\|P_mT^jP_lx\|\le \frac{\beta_l}{4}\left(\prod_{i=\Delta^{(k)}-j+1}^{\Delta^{(k)}-1}|w_i^{(k)}|\right) \|P_lx\|\le \frac{\beta_l}{4}\|P_lx\|.\]

Hence, we have for all $n$ and $j\le \frac{3}{4}\Delta^{(n+1)}$,
\begin{align*}
\|\pi_{K_n}T^j (I-\pi_{K_n})x\|
&\le \sum_{m< 2^n}\sum_{l\ge 2^n}\|P_mT^j P_lx\|\\
&\le \sum_{m< 2^n}\sum_{l\ge 2^n} \frac{\beta_l}{4}\,\|P_lx\|\\
&\le 2^n \Bigl(\sum_{l\ge 2^n}\frac{\beta_l}{4}\Bigr)\,\|(I-\pi_{K_n})x\|\\
&\le 2^n \Bigr(\sum_{k\ge n+1}2^{k-1}\gamma_k\Bigr)\,\|(I-\pi_{K_n})x\|\le \|(I-\pi_{K_n})x\|.
\end{align*}

So, if we take care to ensure that $\frac{3}{4}\Delta^{(n+1)}\ge (n+1)^{2^{J_n}}J_n=(n+1)^{2^{2\Delta^{(n)}}}2\Delta^{(n)}$ for all $n\ge 1$, then condition (b) in  Theorem~\ref{nothfhc} is satisfied. This concludes the proof of Theorem \ref{Ctypeex}.
\section{Extending frequently $d$-hypercyclic tuples}\label{extend}

Let us recall the definition of $d\,$-$\mathcal F$-hypercyclicity, for a given Furstenberg family $\mathcal F\subset 2^\NN$: a tuple of operators $(T_1, \dots ,T_N)$ is $d\,$-$\mathcal F$-hypercyclic if there exists $x\in X$ such that $x\oplus\cdots \oplus x$ is $\mathcal F$-hypercyclic for $T_1\oplus\cdots \oplus T_N$.

\smallskip
In this section, our aim is to prove the following result, which is  a natural analogue of \cite[Theorem 2.1]{MaSa24} for $d\,$-$\mathcal F$-hypercyclicity.  Let us denote by SOT the \emph{Strong Operator Topology} on $\mathfrak L(X)$, \textit{i.e.} the topology of pointwise convergence.

\begin{theorem}\label{thm:extendingdfhc2}
Let $\mathcal F\subset 2^\NN$  be a Furstenberg family, and let $X$ be a Banach space supporting a hereditarily $\mathcal F$-hypercyclic operator. Let $T_1,\dots,T_N\in\mathfrak L(X)$, and assume that $(T_1,\dots ,T_N)$ is $d$-$\mathcal F$-hypercyclic. Then, for any countable and linearly independent set $Z\subset\deffhc(T_1,\dots ,T_N)$, the set
$$\bigl\{T\in\mathfrak L(X):\ Z\subset d\hbox{-}\mathcal F\hbox{-\rm HC}(T_1,\dots,T_N,T)\bigr\}$$
is {\rm SOT}-dense in $\mathfrak L(X).$
\end{theorem}

\smallskip
Applying this result to $Z=\{x\}$ with $x\in \deffhc(T_1,\dots,T_N)$, we get 
\begin{corollary}\label{thm:extendingdfhc} Let $X$ be a Banach space supporting a hereditarily frequently hypercyclic operator.
Let $T_1,\dots,T_N\in\mathfrak L(X)$, and assume that $(T_1,\dots ,T_N)$ is $d$-frequently hypercyclic.
Then there exists $T_{N+1}\in\mathfrak L(X)$ such that $(T_1,\dots,T_{N+1})$ is $d$-frequently hypercyclic.
\end{corollary}
If $(T_1,\dots ,T_N)$ is \emph{densely} $d$-$\mathcal F$-hypercyclic then, applying Theorem \ref{thm:extendingdfhc2} with any dense linearly independent set $Z\subset X$
 contained in $\deffhc(T_1,\dots,T_N)$, we obtain:

\begin{corollary}\label{thm:extendingddensefhc}
Let $\mathcal F\subset 2^\NN$  be a Furstenberg family, and let $X$ be a Banach space supporting a hereditarily $\mathcal F$-hypercyclic operator. Let $T_1,\dots,T_N\in\mathfrak L(X)$, and assume that 
$(T_1,\dots ,T_N)$ is densely $d$-$\mathcal F$-hypercyclic. 
Then the set
$$\bigl\{T\in\mathfrak L(X):\ (T_1,\dots,T_N,T)\;\hbox{ is densely $d$-$\mathcal F$-hypercyclic}\bigr\}$$
is {\rm SOT}-dense in $\mathfrak L(X).$
\end{corollary}

\smallskip In the proof of Theorem \ref{thm:extendingdfhc2}, we will need the following fact (already mentioned in the introduction).
\begin{proposition}\label{densely} If $T\in\mathfrak L(X)$ is hereditarily $\mathcal F$-hypercyclic then it is in fact \emph{densely} hereditarily $\mathcal F$-hypercyclic: 
given a countable family $(A_i)_{i\in I}\subset\mathcal F$ and a family $(V_i)_{i\in I}$ of non-empty open sets in $X$, there is a dense set of $x\in X$ such that $\mathcal N_T(x,V_i)\cap A_i\in\mathcal F$ for all $i\in I$.
\end{proposition}
\begin{proof} Let $U$ be a non-empty open set in $X$: we want to find $x\in U$ such that $\mathcal N_T(x,V_i)\cap A_i\in\mathcal F$ for all $i\in I$. For $(i,N)\in I\times \NN$, define $V_{i,N}:=T^{-N}(V_i)$ and $A_{i,N}:= A_i$. Since $T$ is hereditarily $\mathcal F$-hypercyclic, one can find $x_0\in X$ and sets $B_{i,N}\in\mathcal F$ such that $B_{i,N}\subset A_{i,N}=A_i$ and $T^nx_0\in V_{i,N}$ for all $n\in B_{i,N}$. Since we may assume that the family $(V_i)$ is a basis of open sets for $X$, the vector $x_0$ is in particular a hypercyclic vector for $T$. So one can find an integer $N_U$ such that $x:= T^{N_U}x_0\in U$. Then, for all $n\in B_{i,N_{U}}$, we see that $T^nx= T^{N_U}(T^nx_0)\in V_i$.
\end{proof}

\begin{proof}[Proof of Theorem \ref{thm:extendingdfhc2}] Let us denote by $\glx$ the set of all invertible operators on $X$. The core of the proof is contained in the following fact.
\begin{fact}\label{truc} Let $T_1,\dots,T_N\in\mathfrak L(X)$, and assume that
$(T_1,\dots ,T_N)$ is $d$-$\mathcal F$-hypercyclic.  Let $T\in\mathfrak L(X)$ be a hereditarily $\mathcal F$-hypercyclic  operator. For any countable and linearly independent set $Z\subset\deffhc(T_1,\dots ,T_N)$, for any $S\in\glx$ and any $\varepsilon>0$, one can find $L\in\glx$ such that $\Vert L-S\Vert<\varepsilon$ and $Z\subset \deffhc(T_1,\dots ,T_N, L^{-1}TL)$.
\end{fact}
\begin{proof}[Proof of Fact \ref{truc}]
Without loss of generality, we assume that $Z$ is infinite; we enumerate $Z$ as a sequence $(z_l)_{l\in\NN}$, without repetition. Let us fix $S\in\glx$ and $\varepsilon>0$. Since $\glx$ is $\Vert\,\cdot\,\Vert$-open in $\mathfrak L(X)$, we may assume that any operator $L\in\mathfrak L(X)$ such that $\Vert L-S\Vert <\varepsilon$ is invertible.

Let $(W_p)_{p\in\NN}$ be a countable basis of open sets for $X^{N+1}=X^N\times X$, and assume that each set $W_p$ has the form $W_p=U_p\times V_p$ where $U_p$ is open in $X^N$ and $V_p$ is open in $X$.
For each $(l,p)\in \NN\times \NN$, we may fix a set $A_{l,p}\in\mathcal F$  such that \[ (T_1^n z_l,\dots,T_N^n z_l)\in U_p\qquad\hbox{ for all $n\in A_{l,p}$.}\]
 
We construct by induction a sequence $(L_l)_{l\geq 0}$ in $\mathfrak L(X)$ with $L_0=S$, a sequence $(x_l)_{l\geq 1}$ of vectors of $X$ and a family $(B_{l,p})_{l,p\geq 1}$ of sets in $\mathcal F$, such that the following holds true for every $l,p\geq 1$.
\begin{enumerate}[(i)]
\item $B_{l,p}\subset A_{l,p}$ and $T^n x_l\in V_p$ for all $n\in B_{l,p}$;
\item $L_l(z_s)=x_s$ for all $s\leq l$;
\item $\Vert L_l-L_{l-1}\Vert < 4^{-l}\varepsilon$.
\end{enumerate}

\smallskip The inductive step is as follows. Choose a linear functional $v_l^*\in X^*$ such that $v_l^*(z_s)=0$ for all $s<l$ and $v_l^*(z_l)=1$, which is possible by linear independence of $Z$. Next, since $T$ is \emph{densely} hereditarily $\mathcal F$-hypercyclic by Proposition \ref{densely}, we can find a vector $x_l\in X$ and sets $B_{l,p}\subset A_{l,p}$ with $B_{l,p}\in\mathcal F$
for each  $p\geq 1$,  such that
\[ \|x_l-L_{l-1}(z_{l})\|< \frac \veps{4^l\,\|v_l^*\|}\qquad{\rm and}\qquad T^{n}x_l\in V_p\quad\hbox{for all $n\in B_{l,p}$}.\]

Then, define $L_l:=L_{l-1}+v_l^*\otimes\bigl(x_l-L_{l-1}(z_l)\bigr)\in\mathfrak L(X)$, \textit{i.e.}  
$$L_l(x)=L_{l-1}(x)+v_l^*(x)\,\bigl(x_l-L_{l-1}(z_l)\bigr).$$
Clearly, $L_l(z_s)=L_{l-1}(z_s)$ for all $s<l$, so that $L_{l}(z_s)=x_s$ by the induction hypothesis, $L_l(z_l)=x_l$ and 
$\|L_l-L_{l-1}\|<4^{-l}\veps.$

\smallskip
By (ii) and (iii), the sequence $(L_l)$ converges to some $L\in\mathfrak L(X)$ which satisfies $L(z_l)=x_l$ for all $l\in\NN$
and $\|L-S\|<\veps$; in particular, $L$ is invertible. Moreover, $T^n x_l\in V_p$ for all $l,p\geq 1$ and $n\in B_{l,p}$. Since $B_{l,p}\subset A_{l,p}$, it follows that 
\[ (T_1^nz_l, \dots , T_N^n z_l, (L^{-1}TL)^nz_l)\in U_p\times L^{-1}(V_p)=:\widetilde{W_p}\qquad\hbox{for all $n\in B_{p,l}$}.\]

Now, $( \widetilde{W_p})_{p\in\NN}$ is a basis of the topology of $X^{N+1}= X^N\times X$ because $I\oplus L^{-1}$ is a homeomorphism of $X^{N+1}$; so we see that $z_l\in \deffhc(T_1,\dots, T_N, L^{-1}TL)$ for each 
$l\geq 1$.
\end{proof}

To conclude the proof of Theorem \ref{thm:extendingdfhc2}, we observe that since the operator $T$ is hypercyclic, the similarity orbit of $T$, \textit{i.e.} the set $\{S^{-1}TS:\ S\in\glx\}$, is SOT-dense in $\mathfrak L(X)$; see e.g. \cite[Proposition 2.20]{BM09}. By Fact \ref{truc}, it follows 
that the set
$$\bigl\{L^{-1}TL:\ Z\subset \deffhc(T_1,\dots,T_N,L^{-1}TL),\ L\in\glx\bigr\}$$
is SOT-dense in $\mathfrak L(X)$.
\end{proof}

\begin{remark} Our proof of Theorem \ref{thm:extendingdfhc2} differs from that of \cite[Theorem 2.1]{MaSa24} regarding $d$-hypercyclicity, where a Baire category argument was used; and it must be so at least for $d$-frequent hypercyclicity, since $\mathrm{FHC}(T)$ is always meager in $X$, for any operator $T\in\mathfrak L(X)$ (see \cite{Moo} or \cite{BAYRUZSA}). However, there may be a Baire category proof of Theorem \ref{thm:extendingdfhc2} when $\mathcal F$ is the family of sets with positive upper density (or, more generally, an ``upper'' Furstenberg family in the sense of \cite{BoGre18}).
\end{remark}

To apply Theorem \ref{thm:extendingdfhc2}, in the frequently hypercyclic case, it would be nice to exhibit a class of Banach spaces as large as possible
supporting hereditarily frequently hypercyclic operators.
It is easy to see that  any Banach space with a symmetric Schauder basis has this property: it suffices to take
$T:=2B$ where $B$ is the backward shift associated to the basis. In view of Corollary \ref{USD}, a natural (and much broader) class would be that of complex Banach spaces admitting an unconditional
Schauder decomposition, but we are not able to prove that every such space has the required property. In any event, we can use the method of \cite{Shk10c} to prove the existence of $d$-frequently hypercyclic tuples of arbitrary length for this class of spaces.

\begin{proposition}
 Let $X$ be a complex separable infinite-dimensional Banach space with an unconditional Schauder decomposition. For any $N\geq 1$, there exist $T_1,\dots,T_N\in\mathfrak L(X)$ such that $(T_1,\dots,T_N)$ is $d$-frequently hypercyclic.
\end{proposition}
\begin{proof}
 By \cite{DFGP12}, $X$ supports an operator $T$ with a perfectly spanning set of $\TT$-eigenvectors. Then $T\oplus\cdots \oplus T$ has the same property, and in particular it is frequently hypercyclic; let $x_1\oplus\cdots\oplus x_N$ be a frequently hypercyclic vector for $T\oplus  \cdots \oplus T.$ 
 Now, let $y\in X\setminus\{ 0\}$ be arbitrary. Since $\glx$ acts transitively on $X$, we may choose $S_1,\dots ,S_N\in\glx$ such that $S_i(x_i)=y$ for $i=1,\dots, N$. Then, setting $T_i:=S_i T S_i^{-1}$, we see that $y\in \dfhc(T_1,\dots,T_N)$. 
\end{proof}

\section{Frequent $d$-hypercyclicity vs dense frequent $d$-hypercyclicity}\label{pasdensely}

Despite the similarity of the definitions, there are strong differences between hypercyclicity and $d$-hypercyclicity. For instance, if $T\in\mathfrak L(X)$ is hypercyclic then  
$\hc(T)$ is always dense in $X$, and $\hc(T)\cup\{ 0\}$ contains a dense linear subspace of $X$. On the contrary, for two operators $T_1$ and $T_2,$ the set $\dhc(T_1,T_2)\cup\{0\}$ may be equal to some finite-dimensional subspace  (see \cite[Theorem 3.4]{SaSh14}). In particular, $d$-hypercyclic tuples are not necessarily ``densely $d$-hypercyclic''. 

On the other hand, since frequent $d$-hypercyclicity is a strong form of $d$-hypercyclicity, it is natural to ask whether some properties
that are not true for  $d$-hypercyclic tuples might be true for $d$-frequently hypercyclic tuples. In this spirit, the following question was asked in \cite{MaSa16},  \cite{MaPu21} and \cite{MaMePu}.
\begin{question}
Let $(T_1,T_2)$ be a $d$-frequently hypercyclic pair of operators on a Banach space $X$. Is $(T_1,T_2)$ necessarily densely $d$-hypercyclic?
\end{question}
The next theorem provides a solution to this problem.
\begin{theorem}\label{thm:dfhcnotddens}
There exist a Banach space $X$ and $T_1,T_2\in\mathfrak L(X)$ such that $(T_1,T_2)$ is $d$-frequently hypercyclic but not densely $d$-hypercyclic.
\end{theorem}

\smallskip
Our proof is inspired by \cite{SaSh14}, where the authors construct a $d$-hypercyclic pair which is not densely $d$-hypercyclic. A key role will be played by the {similarity orbit} of some well chosen operator $T$. The next two lemmas point out the relationship between
(frequently) hypercyclic vectors of $T$ and $T\oplus T$ and (frequently) $d$-hypercyclic vectors of $(T_1,T_2)$ when $T_1$ and $T_2$ belong to the similarity orbit of $T$.

\smallskip
\begin{lemma}\label{lem:conjugate1}
Let $T\in\mathfrak L(X)$ and $L_1,L_2\in\glx$, and set $T_m:=L_m^{-1}TL_m,$ 
$m=1,2$. Let also $x\in X$. If $x\in \dhc(T_1,T_2)$ then $L_2x-L_1x\in \hc(T)$. 
\end{lemma}
\begin{proof}
This is \cite[Lemma 3.1]{SaSh14}.
\end{proof}

\smallskip
\begin{lemma}\label{lem:conjugate2}
Let $T\in\mathfrak L(X)$ and $L_1,L_2\in\glx$, and set $T_m:=L_m^{-1}TL_m$.  Let also $\mathcal F\subset 2^\NN$ be a Furstenberg family, and let $x\in X$. Then $x\in \deffhc(T_1,T_2)$  if and only if $(L_1x,L_2x)\in \effhc(T\oplus T)$. 
\end{lemma}
\begin{proof} It is identical to the proof of \cite[Lemma 2.1]{SaSh14}.  Just observe that for any positive integer $n$ and any pair of non-empty open subsets $(U, V)$ in $X$,
$$
\bigl(T^{n}L_1x, T^nL_2x)\in U\times V\iff \bigl( T_1^nx, T_2^n x)\in L_1^{-1}(U)\times L_2^{-1}(V),$$
and apply the definition of $\mathcal F$-hypercyclicity.
\end{proof}

\smallskip The operators $T_1$ and $T_2$ that we are going to construct  will be such that $T_2=(cI+R)^{-1}T_1(cI+R)$ for some $R\in\mathfrak L(X)$ and $c>0$. By Lemma \ref{lem:conjugate1} (with $T=T_1$ and $L_1=cI$),
any $x\in\dhc(T_1,T_2)$ is such that $Rx\in \hc(T_1)$. We shall construct $T_1$ and $R$ in such a way that this condition prevents $\dhc(T_1,T_2)$ from being dense in $X$. 
The following result will be useful to prove that the pair $(T_1,T_2)$ is $d$-frequently hypercyclic.

\begin{lemma}\label{lem:conjugate3}
Let $T_1,R\in\mathfrak L(X)$ and $c>0$ be such that $L_2=cI+R$ is invertible, and let $T_2:=L_2^{-1}T_1L_2$. Let also $\mathcal F\subset 2^\NN$ be a Furstenberg family. If  $x\in X$ is such that $(x,Rx)\in\effhc(T_1\oplus T_1)$, then 
 $x\in \deffhc(T_1,T_2)$.
\end{lemma}
\begin{proof}
Let us set $L_1:=cI.$ By Lemma \ref{lem:conjugate2}, 
it suffices to show that the condition $(x,Rx)\in \effhc(T_1\oplus T_1)$ implies $(L_1x,L_2x)\in \effhc(T_1\oplus T_1)$. Now it is clear that $(cx,Rx)\in \effhc(T_1\oplus T_1)$. Let $U,$ $V$ be two non empty open subsets of $X$ and let $U'\subset U$ and $W$ be non empty open sets such that $U'+W\subset  V.$ There exists a set $A\in\mathcal F$ such that $T_1^{n}(cx)\in U'$ and $T_1^{n}(Rx)\in W$ for all $n\in A$. Then, for all $n\in A$, we have
$$T_1^n L_1x=T_1^n(cx)\in U\qquad{\rm and}\qquad T_1^{n}L_2x=T_1^{n}(cx+Rx)\in U'+ W\subset V.$$
\end{proof}

We now go into the details of the construction. First, we define the Banach space $X$ as
\[ X:=\left( \bigoplus_{l\geq 1} X(l)\right)_{c_0}\quad \hbox{where}\quad\hbox{$X(l)=\ell_1(\ZZ_+)$ for every $l\ge 1$}.
\]
(Following a standard notation, the subscript ``$c_0$'' indicates that the direct sum  is a $c_0$-sum.)

Next, we introduce the following operator $T\in \mathfrak L(X)$: denoting by $B$ the canonical backward shift on $\ell_1(\ZZ_+)$, let
\[ T:=\bigoplus_{l\geq 1} T(l)\quad\hbox{where}\quad T(l)=I+2^{-l}B\in\mathfrak L\bigl(X(l)\bigr)\textrm{ for every }l\ge 1.\]

\begin{lemma}\label{lem:ToplusTfhc}
The operator $T\oplus T$ is frequently hypercyclic.
\end{lemma}
\begin{proof}
It is enough to prove that $T$ has a perfectly spanning set of $\TT$-eigenvectors.
Indeed, $T\oplus T$ will have the same property and hence will be frequently hypercyclic.
\end{proof}

\smallskip We now define our first operator $T_1$:
 \[ T_1:=\bigoplus_{l\geq 1} (I+B)\in\mathfrak L(X).\] 
 
Note that the same proof as that of Lemma \ref{lem:ToplusTfhc} shows that $T_1\oplus T_1$ is frequently hypercyclic. However, we will use the above operator $T$ to produce a frequently hypercyclic vector for $T_1\oplus T_1$ with specific properties.

\smallskip
 In what follows, we denote by $(e_k(l))_{k\geq 0}$ the canonical basis of the $l$-th component $X(l)=\ell_1(\ZZ_+)$ of $X$, and by 
$(e_k^*(l))_{k\geq 0}$ the associated sequence of coordinate functionals, which we will consider as linear functionals on $X$. A vector $x\in X$ will be written as $x=(x(l))_{l\geq 1}$, and we will use the notation $x_k(l)=\langle e_k^*(l), x\rangle$ for every $k\ge 0$.

\begin{lemma}\label{lem:specificfhc}
There exists $(u,v)\in \fhc(T_1\oplus T_1)$ such that $u_0(1)\neq 0$ and 
$|v_k(l)|\leq 2^{-lk}$ for all $k\geq 0$ and all $l\geq 1$.
\end{lemma}
\begin{proof}
For $l\geq 1,$ Let us consider the diagonal operator $D(l)$ on $X(l)$ defined by 
$$D(l)(x(l)):=\sum_{k\geq 0}2^{-lk}x_k (l)e_k(l),$$
and set \[ D:=\bigoplus_{l\geq 1} D(l)\in\mathfrak L(X).\]

It is easy to check
that $(I+B)D(l)=D(l)(I+2^{-l}B)$ for each $l\geq 1$, so that $T_1D=DT$. Moreover, the operator $D$ has dense range. So $T_1$ is a quasi-factor of $T$ with quasi-factoring map $D$, and hence $T_1\oplus T_1$ is a quasi-factor of $T\oplus T$ with quasi-factoring map $D\oplus D$. Let $(x,y)\in \fhc(T\oplus T)$ with $\|y\|\leq 1$ and $x_0(1)\neq 0$
and let us set $(u,v):=(Dx,Dy)$. Then $(u,v)\in \fhc(T_1\oplus T_1)$. Moreover, $u_0(1)=x_0(1)\neq 0$ and, for $k\geq 0$ and $l\geq 1,$ 
$$|v_k(l)|=|2^{-lk}y_k(l)|\leq 2^{-lk}.$$
\end{proof}

We now give a result which provides a condition preventing a vector from being hypercyclic for $T_1$.
\begin{lemma}\label{lem:nohc}
Let $x\in X$. Assume that there exist $l\geq 1$  and $\lambda\neq 0$ such that, for all $k\geq 1$ sufficiently large, $\Re e\,\big\langle e_k^*(l), x/\lambda\big\rangle\geq 0.$ Then
$x\notin \hc(T_1)$.
\end{lemma}
\begin{proof}
Since $x\in \hc(T_1)$ if and only if $x/\lambda\in {\rm HC}(T_1)$, we may assume $\lambda=1$. 
Now if $\Re e (x_k(l))=\Re e\,\big\langle e_k^*(l), x\big\rangle\geq 0$ for all sufficiently large $k,$ 
then the arguments of \cite{SaSh14} (see the proof of Claim 1 page 845) show that
either $x(l)$ is finitely supported or $$\Re e\, \Big\langle e_0^*(l), (I+B)^n(x(l))\Big\rangle\geq 0\quad\hbox{for all sufficiently large $n$.}$$
 Therefore
$x(l)$ cannot be a hypercyclic vector for $I+B$, and hence $x\notin \hc(T_1)$.
\end{proof}

\smallskip
Let us fix   a sequence of positive real numbers $(\veps_l)_{l\geq 1}$ going to zero. 
Let also $V:\ell_1(\ZZ_+)\to\ell_1(\ZZ_+)$ be the (bounded) operator defined by
\[ Vy:= \sum_{k\geq 1}\left(\sum_{j\geq 1}2^{-jk}y_j\right)e_k \quad\textrm{ for every }y\in \ell_1(\ZZ_+)\]
and let $R_0:X\to X$ be the (bounded) operator on $X$ defined by
$$R_0(x):=\bigl(\veps_1 V(x(1)),\dots,\veps_l V(x(1)),\dots\bigr).$$

The operator $R_0$ satisfies the following crucial estimates: 
\begin{lemma}\label{lem:crucialestimate}
Let $x\in X$ and $\ m\geq 1$ be such that $\langle e_m^*(1),x\rangle\neq 0$ and $\langle e_j^*(1),x\rangle=0$ for $1\leq j<m.$ Then there exists $\delta:\ZZ_+\longrightarrow \RR_+$ such that $\delta(k)\longrightarrow 0$ as $k\rightarrow \infty$ and such that for all $l\geq 1$ and $k\ge 0$ we can write
$$\bigl\langle e_k^*(l), R_0(x)\big\rangle=\veps_l\, \langle e_m^*(1),x\rangle\,2^{-mk}(1+\delta(k)).$$
\end{lemma}
\begin{proof}
We just write 
\begin{align*}
\bigl\langle e_k^*(l), R_0(x)\big\rangle&=\veps_l \sum_{j=1}^{\infty}2^{-jk}\langle e_j^*(1),x\rangle\\
&=\veps_l \sum_{j=m}^{\infty}2^{-jk}\langle e_j^*(1),x\rangle\\
&=\veps_l 2^{-mk}\langle e_m^*(1),x\rangle\left(1+\sum_{j=1}^{\infty}2^{-jk}\frac{\langle e_{j+m}^*(1),x\rangle}{\langle e_m^*(1),x\rangle}\right);
\end{align*}
and we conclude because
$$\left|\sum_{j=1}^{\infty}2^{-jk}\frac{\langle e_{j+m}^*(1),x\rangle}{\langle e_m^*(1),x\rangle}\right|\leq \frac{2^{-k}}{\langle e_m^*(1),x\rangle}\|x(1)\|_1\xrightarrow{k\to\infty}0.$$
\end{proof}

We consider $(u,v)$ given by Lemma \ref{lem:specificfhc}. We set, for $x\in X,$
$$R(x):=R_0(x)+\frac{x_0(1)}{u_0(1)}(v-R_0(u)).$$
This defines a bounded operator such that $R(u)=v\in {\rm HC}(T_1).$ It turns out that there are not so many vectors $x\in X$ such that $R(x)\in {\rm HC}(T_1)$.
\begin{lemma}\label{lem:restrictionR}
Let $x\in X$ be such that $R(x)\in {\rm HC}(T_1)$. Then $x(1)$ is a scalar multiple of $u(1)$.
\end{lemma}
\begin{proof}
Let us set $z:=x-\frac{x_0(1)}{u_0(1)}u$, 
so that
$$R(x)=R_0(z)+\frac{x_0(1)}{u_0(1)}v.$$
Assume first that $z(1)\notin \mathrm{span}(e_0(1))$. Then, there exists $m\geq 1$ such that 
$\langle e_m^*(1), z\rangle\neq 0$ whereas $\langle e_j^*(1), z\rangle=0$ for $1\leq j<m$. Let $l>m$. By Lemma \ref{lem:crucialestimate}, it follows that
$$\langle e_k^*(l), R_0(z)\rangle=\veps_l \langle e_m^*(1), z\rangle 2^{-mk}(1+\delta(k)),$$
so that
\begin{align*}
\langle e_k^*(l), R(x)\rangle&=\veps_l \langle e_m^*(1), z\rangle 2^{-mk}(1+\delta(k))+\frac{x_0(1)}{u_0(1)}v_k(l)\\
&=\veps_l \langle e_m^*(1)(z)2^{-mk}(1+\delta_l(k)),
\end{align*}
where for all $l>m$, $\delta_l(k)\rightarrow 0$ as $k\rightarrow \infty$,
since $|v_k(l)|\leq 2^{-lk}=o(2^{-mk})$. By Lemma \ref{lem:nohc}, $R(x)\notin {\rm HC}(T_1)$, a contradiction.

Hence, there exists a complex number $\alpha(1)$ such that
$$x(1)-\frac{x_0(1)}{u_0(1)}u(1)=\alpha(1)e_0(1).$$
Applying the functional $e_0^*(1)$ to this equation, we easily get $\alpha(1)=0$, which implies that $x(1)$ belongs to $\mathrm{span}(u(1))$.
\end{proof}

\smallskip We can now give the
\begin{proof}[Proof of Theorem \ref{thm:dfhcnotddens}] Let us fix $c>\|R\|$. We set $L_2:=cI+R$ (which is invertible) and $T_2:=L_2^{-1}TL_2$. We show that the pair 
$(T_1,T_2)$ is d-frequently hypercyclic but not densely d-hypercyclic.

\smallskip By construction, $(u,Ru)=(u,v)\in \fhc(T_1\oplus T_1)$. Hence $u\in \dfhc(T_1,T_2)$ by Lemma~\ref{lem:conjugate3}. 
Moreover, setting $T=T_1$ and $L_1=cI,$ Lemma \ref{lem:conjugate1} implies that if $x\in\dhc(T_1,T_2),$ then $R(x)\in {\rm HC}(T_1)$. By Lemma \ref{lem:restrictionR}, it follows that 
$x(1)\in\mathrm{span}(u(1))$ for every $x\in\dhc(T_1,T_2)$. In particular, $\dhc(T_1,T_2)$ cannot be dense in $X$.
\end{proof}

\section{Eigenvectors and $d$-frequent hypercyclicity}\label{ajout}

In this section, we give a  criterion relying on properties of the eigenvectors for showing that a tuple of operators is (densely) $d$-frequently hypercyclic. The initial motivation was the following question asked by K. Grosse-Erdmann: let $D$ be the derivation operator acting on the space of entire functions $H(\CC)$, and for every $a\in\CC\setminus\{ 0\}$, denote by $\tau_a$  the operator of translation by $a$ on $H(\CC)$, defined by $\tau_a f(z):= f(z+a)$. It is well-known (see \cite{BM09} or \cite{KarlAlfred}) that both $D$ and $\tau_a$ are frequently hypercyclic. Now one can ask

\begin{question}\label{Dtau}
Do the operators $D$ and $\tau_a$ have common frequently hypercyclic vectors?
\end{question}

It will follow from the next theorem that the answer to Question \ref{Dtau} is positive.

\begin{theorem}\label{tambouille} Let $N\geq 2,$ let $X$ be a complex Fr\'echet space, and let $T_1, \dots ,T_N\in\mathfrak L(X)$. Assume that there exist a holomorphic vector field $E:O \to X$ defined on some connected open set $O\subset\CC$, and non-constant holomorphic functions $\phi_1, \dots,\phi_N$ defined on some connected open set containing $O$, such that 
\begin{itemize}
\item[-] $\overline{\rm span}\, E(O)=X$;
\item[-] $T_i E(z)=\phi_i(z) E_i(z)$ for every $i=1,\dots ,N$ and $z\in O$;
\item[-] $O\cap \phi_i^{-1}(\TT)\cap \bigcap_{j\neq i} \phi_j^{-1}(\DD)\neq\emptyset$ for every $i=1,\dots ,N$.
\end{itemize}

Then the $N$-tuple $(T_1,\dots, T_N)$ is densely $d$-frequently hypercyclic.
\end{theorem}

\smallskip Before proving this result, let us state two consequences and give some examples.
\begin{corollary}\label{cormachin} Let $D$ be the derivation operator on $X:=H(\CC)$. If $\phi_1$ and  $\phi_2$ are two entire functions of exponential type such that $\phi_1^{-1}(\TT)\cap \phi_2^{-1}(\DD)\neq\emptyset$ and $\phi_2^{-1}(\TT)\cap \phi_1^{-1}(\DD)\neq\emptyset$, then the pair $(\phi_1(D), \phi_2(D))$ is densely $d$-frequently hypercyclic.
\end{corollary}
\begin{proof} Let $E:\CC\to X$ be the holomorphic vector field defined by $E(z):= e^{z\,\cdot\, }$. We have $DE(z)=zE(z)$ for all $z\in\CC$ and $\overline{\rm span}\, E(\CC)=X$; so we may apply Theorem \ref{tambouille} to the operators $T_i:=\phi_i(D)$. 
\end{proof}

Since $\tau_a=\Phi_a(D)$, where $\phi_a(z):=e^{az}$, Corollary \ref{cormachin} applies to pairs of operators involving $D$ and $\tau_a$:

\begin{example} Taking $\phi_1(z):= z$ and $\phi_2(z):=e^{az}$, we see that for  any $a\neq 0$, the pair $(D,\tau_a)$ is densely $d$-frequently hypercyclic (so that in particular $\fhc(D)\cap \fhc(\tau_a)\neq \emptyset$).
Indeed, any complex number $z$ such that $\vert z\vert=1$ and $\Re e(az)<0$  belongs to $\phi_1^{-1}(\TT)\cap \phi_2^{-1}(\DD)$, while
any
$z\in\ i\overline{a}\,\RR$ such that $|z|<1$ belongs to  $\phi_2^{-1}(\TT)\cap \phi_1^{-1}(\DD)$.
Similarly, if $a,b\neq 0$ and $a/b\not\in \RR$ then $(\tau_a,\tau_b)$ is densely $d$-frequently hypercyclic.
\end{example}

\smallskip
\begin{corollary} Let $B$ be the canonical backward shift acting on $X=\ell_p(\ZZ_+)$ or $c_0(\ZZ_+)$. If $\phi_1$ and  $\phi_2$ are two holomorphic functions defined in a neighbourhood of the closed unit disk $\overline\DD$ such that $\DD\cap \phi_1^{-1}(\TT)\cap \phi_2^{-1}(\DD)\neq\emptyset$ and $\DD\cap \phi_2^{-1}(\TT)\cap \phi_1^{-1}(\DD)\neq\emptyset$, then the pair $(\phi_1(B), \phi_2(B))$ is densely $d$-frequently hypercyclic. 
\end{corollary}
\begin{proof} The operators $T_i:=\phi_i(B)$ are well defined since $\sigma(B)=\overline{\,\DD}$. Let $E:\DD\to X$ be the holomorphic vector field defined by $E(z):=\sum_{n=0}^\infty z^n e_n$. We have $BE(z)=zE(z)$ for all $z\in\DD$ and $\overline{\rm span}\, E(\CC)=X$; so Theorem \ref{tambouille} applies.
\end{proof}

\begin{example} If $\vert \lambda\vert >1$ and $0<\vert\alpha\vert<2\vert\lambda\vert$, the pair $(\lambda B, I+\alpha B)$ is densely $d$-frequently hypercyclic. If $\vert\lambda\vert >1$ and $a\neq 0$, the pair $(\lambda B, e^{aB})$ is densely $d$-frequently hypercyclic. On the other hand, Theorem \ref{tambouille} is completely inefficient to show for example that the pair $(aB, bB^2)$ is $d$-frequently hypercyclic if $1<a<b$, which is nevertheless true by \cite{MaMePu}.
\end{example}

\smallskip In the proof of Theorem \ref{tambouille}, we will make use of the following straightforward observation.
\begin{fact}\label{seriously?} Let $T_1,\dots ,T_N\in\mathfrak L(X)$ and let $x_1,\dots ,x_N\in X$. Assume that $(x_1,\dots ,x_N)\in \fhc(T_1\oplus\cdots \oplus T_N)$ and that $T_j^m x_i\to 0$ as $m\to \infty$ whenever $i\neq j$. Then $x:=x_1+\cdots +x_N$ is a $d$-frequently hypercyclic vector for $(T_1,\dots ,T_N)$.
\end{fact}

\smallskip
\begin{proof}[Proof of Theorem \ref{tambouille}] 
We first note that for $i=1,\dots ,N$, there is a non-empty open set $V_i\subset O$ and $r_i\in (0,1)$ such that $(\phi_i)_{| V_i}$ is a diffeomorphism (onto its range), $\phi_i(V_i)\cap\TT\neq\emptyset$ and $\phi_j(V_i)\subset D(0,r_i)$ for $j\neq i$. Indeed, let $a\in O$ be such that $\phi_i(a)\in\TT$ and $\phi_{j}(a)\in\DD$ for $j\neq i$. Choose an open neighbourhood $W$ of $a$ and $r_i\in (0,1)$ such that $\phi_{j}(W)\subset D(0,r_i)$ for all $j\neq i$. By the open mapping theorem, $\phi_i(W)$ is an open set intersecting $\TT$, so $\phi_i(W)\cap\TT$ is uncountable. Hence, one can find $b\in W$ such that $\phi_i(b)\in\TT$ and $\phi'_i(b)\neq 0$; and the claim follows from the inverse function theorem.

Taking the open set $V_i$ smaller if necessary, we may assume that $\Lambda_i:=\phi_i(V_i)\cap \TT$ is a proper open arc of $\TT$.  We choose a ``cut-off'' function $\chi_i\in\mathcal C^\infty(\TT)$ such that $\chi_i(\lambda)=0$ outside $\Lambda_i$ and $\chi_i(\lambda)>0$ on $\Lambda_i$, and (with the obvious abuse of notation) we define $F_i:\TT\to X$ by setting
\[ F_i(\lambda):=\chi_i(\lambda) E\bigl( \phi_{i}^{-1}(\lambda)\bigr)\quad\textrm{ for every }\lambda\in\TT.\]

Thus, $F_i$ is a $\mathcal C^\infty$-$\,$smooth $\TT$-eigenvector field for $T_i$, \textit{i.e.} $T_iF_i(\lambda)=\lambda F_i(\lambda)$ for every $\lambda\in\TT$. Let us denote by $\widehat{F_i}(n)$ the Fourier coefficients of $F_i$:
\[ \widehat{F_i}(n)=\int_\TT F_i(\lambda)\, \lambda^{-n}\, d\lambda, \quad n\in\ZZ.\]

Since $F_i$ is a $\TT$-eigenvector field for $T$, we have $T_i\widehat{F_i}(n)=\widehat{F_i}(n-1)$ for all $n\in\ZZ$, \textit{i.e.} the sequence $\big(\widehat{F_i}(n)\bigr)_{n\in\ZZ}$ is a bilateral backward orbit for $T_i$. Moreover, since $\overline{\rm span}\,E(O)=X$, it follows from the Hahn-Banach theorem, together with the identity principle for analytic functions, that $\overline{\rm span}\, \bigl( \widehat{F_i}(n)\;:\; n\in\ZZ\bigr)=X$.  

\smallskip In the remainder of the proof, we fix a family $(g_{i,n})_{1\leq i\leq N, n\in\ZZ}$ of independent, standard complex Gaussian variables defined on some probability space $(\Omega,\mathcal A,\mathbb P)$.

\begin{claim}\label{claimtruc} For every $i=1,\dots ,N$, the series $$\sum_{n\in\ZZ} g_{i, n} \widehat{F_i}(n)$$ is almost surely convergent, and defines an $X$-valued random variable
$\xi_i$ on $(\Omega,\mathcal A,\mathbb P)$, which
is such that 
\[ \textrm{for every }i,j=1,\dots ,N \textrm{ with } j\neq i,\; T_j^m\xi_i\xrightarrow{m\to\infty} 0 \quad\hbox{almost surely}.\]
\end{claim}
\begin{proof}[Proof of Claim \ref{claimtruc}] Since $F_i$ is $\mathcal C^2$-$\,$smooth, two integrations by parts show that for any continuous semi-norm $\mathbf q$ on $X$, we have $\mathbf q\bigl( \widehat{F_i}(n)\bigr)=O(1/n^2)$ as $\vert n\vert\to\infty$, so that
\[ \sum_{n=-\infty}^\infty \mathbb E\bigl(\mathbf q(g_{i, n} \widehat{F_i}(n))\bigr)<\infty.\]
This implies that the series $\sum_{n\in\ZZ} g_{i,n} \widehat{F_i}(n)$ is almost surely convergent.

\smallskip Let us fix $j\neq i$. By the definition of $F_i$, we have  $T_jF_i(\lambda)={\psi_{i,j}(\lambda)} F_i(\lambda)$ for every $\lambda\in \Lambda_i$, where ${\psi_{i,j}(\lambda)}:=\phi_j(\phi_i^{-1}(\lambda))$; and $T_jF_i(\lambda) =0$ if $\lambda\not\in\Lambda_i$. Hence, for almost every $\omega\in\Omega$ and every $m\in\NN$, we have
 \[ T_j^m \bigl(\xi_i(\omega)\bigr)=\sum_{n\in\ZZ} g_n(\omega) \int_{\Lambda_i} {\psi_{i,j}(\lambda)}^m F_i(\lambda)\,\lambda^{-n}\, d\lambda.\]
 
Let $\mathbf q$ be a continuous semi-norm on $X$.  Since $\vert {\psi_{i,j}(\lambda)}\vert <r_i$ for every $\lambda\in\Lambda_i$ by definition of $\psi$,  two integrations by parts show that there is a constant $C_{\mathbf q}$ such that
 \[ \mathbf q\left( \int_\Lambda {\psi_{i,j}(\lambda)}^m F(\lambda)\,\lambda^{-n}\, d\lambda\right)\leq C_{\mathbf q}\times \frac{m^2 r_i^m}{1+n^2}\quad\textrm{ for every }m\ge 0 \textrm{ and every }n\in\ZZ.\]
 
Moreover, it follows from the Borel-Cantelli lemma that for almost every  $\omega\in\Omega$, there exists an integer $N(\omega)$ such that  
 \[ \forall \vert n\vert >N(\omega)\;:\; \vert g_n(\omega)\vert \leq \sqrt{n}.\]
 
Hence, given a continuous semi-norm $\mathbf q$ on $X$, one can find for almost every  $\omega\in \Omega$ some constant $M_{\mathbf q,\omega}$ such that 
 \[ \forall m\in\NN\;:\; \mathbf q\bigl( T_j^m\xi_i(\omega)\bigr)\leq M_{q,\omega} \, m^2 r_i^m.\]
 
 Hence $\mathbf q\bigl( T^m_j\xi_i(\omega)\bigr)\to 0$ almost surely as $m\to\infty$  for any given continuous semi-norm $\mathbf q$, \textit{i.e.} $T_j^m\xi_i\to 0$ almost surely.
 \end{proof}

\smallskip We can now conclude the proof of Theorem \ref{tambouille}. For $i=1,\dots ,N$, let us denote by $\mu_i$ the distribution of the random variable $\xi_i:\Omega\to X$. By definition, $\mu_i$ is a $T_i$-invariant Gaussian measure with full support; and by \cite{Kevin}, $T_i$ is \emph{strongly mixing} with respect to $\mu_i$. Hence, the measure $\mu_1\otimes \cdots \otimes \mu_N$ is a $(T_1\oplus\cdots \oplus T_N)\,$-$\,$invariant measure on $X^N$ with full support  and $T_1\oplus\cdots \oplus T_N$ is mixing with respect to $\mu_1\otimes \cdots \otimes \mu_N$. Since $\mu_1\otimes \cdots \otimes \mu_N$ is the distribution of the random vector $\xi:=(\xi_1,\dots ,\xi_N)$ by independence of $\xi_1,\dots ,\xi_N$, it follows that  the vector $\xi(\omega)$ is almost surely frequently hypercyclic for $T_1\oplus\cdots \oplus T_N$. Moreover, by Claim \ref{claimtruc}, we see that $T_j^m \xi_i(\omega)\to 0$ almost surely as $m\to\infty$, whenever $i\neq j$. By Fact \ref{seriously?}, it follows that the vector $\xi_1(\omega)+\cdots +\xi_N(\omega)$ is almost surely $d$-frequently hypercyclic for $(T_1,\dots ,T_N)$. In other words, $(\mu_1\star\cdots\star\mu_N)\,$-$\,$almost every $x\in X$ is $d$-frequently hypercyclic for $(T_1,\dots ,T_N)$. Since the measure $\mu_1\star\cdots\star\mu_N$ has full support, this terminates the proof of Theorem \ref{tambouille}.
\end{proof}

\section{Remarks and questions}\label{quest}

\subsection{Hereditary frequent hypercyclicity in a weak sense}\label{weak}

Another natural definition for hereditary $\mathcal F$-hypercyclicity could be the following: an operator $T\in\mathfrak L(X)$ is \textbf{hereditarily $\mathcal F$-hypercyclic in the weak sense} if, for every $A\in\mathcal F,$ the sequence $(T^n)_{n\in A}$ is $\mathcal F$-hypercyclic, \textit{i.e.} there exists $x\in X$ such that $\mathcal N_T(x,V)\cap A\in\mathcal F$ for all non-empty open sets $V\subset X.$ Equivalently, $T$ is $\mathcal F_A$-hypercyclic for every $A\in\mathcal F$, where $\mathcal F_A:=\{ B\subset\NN\,:\, B\cap A\in\mathcal F\}$. Of course, hereditary $\mathcal F$-hypercyclicity implies hereditary $\mathcal F$-hypercyclicity in the weak sense. Note also that Theorem \ref{Ctypeex} says precisely that there exist frequently hypercyclic operators which are not hereditarily frequently hypercyclic in the weak sense.

\smallskip When $\mathcal F$ is the family of all infinite subsets of $\NN$, an operator $T$ is hereditarily $\mathcal F$-hypercyclic in the weak sense if and only if it is ``hereditarily hypercyclic with respect to the whole sequence of integers'' in the sense of \cite{BP-her}; and this means exactly that $T$ is \emph{topologically mixing} (see e.g. \cite[Lemma 2.2]{Gri05}). The next result shows that this is also equivalent to hereditary $\mathcal F$-hypercyclicity.

\begin{proposition} \label{debile} Let $\mathcal F$ be a Furstenberg family with the following property: for any operator $T$ and any $A\subset \NN$, the set $\mathcal F_A$-{\rm HC}$(T)$ is either empty or comeager in the underlying space. Then, hereditary $\mathcal F$-hypercyclicity and hereditary $\mathcal F$-hypercyclicity in the weak sense are equivalent. In particular,  when $\mathcal F$ is the family of all infinite subsets of $\NN$, an operator $T$ is hereditarily $\mathcal F$-hypercyclic if and only if it is topologically mixing.
\end{proposition}

\begin{proof} Assume that $T$ is hereditarily $\mathcal F$-hypercyclic in the weak sense.
Let $(A_i)_{i\in I}$ be a countable family of sets in $\mathcal F$
and let $(V_i)_{i\in I}$ be a family of non-empty open subsets of $X$. By assumption on $\mathcal F$, for each $i\in I$, the set $G_i$ of $\mathcal F$-hypercyclic vectors for the sequence $(T^n)_{n\in A_i}$ is comeager in $X$; so $G:=\bigcap_{i\in I} G_i$ is non-empty. Then any $x\in G$ satisfies the required property: for every $i\in I$, there is a set $B_i\in\mathcal F$ such that
$B_i\subset A_i$ and
$T^nx\in V_i$ for all $n\in B_i$.
\end{proof}

We can now ask:

\begin{question}\label{question:weaksense}
For which Furstenberg families $\mathcal F$ do hereditary $\mathcal F$-hypercyclicity and hereditary $\mathcal F$-hypercyclicity in the weak sense coincide?
\end{question}

In view of the results of \cite{BoGre18}, \emph{upper Furstenberg families} may be good candidates. However, we are unable to handle even  the case of sets with positive upper density. Proposition \ref{debile} leads naturally to the following question:

\begin{question} Let us denote by $\overline{\mathcal D}$ the family of all sets $A\subset\NN$ with positive upper density. Is it true that if $A\in\overline{\mathcal D}$ and $T\in\mathfrak L(X)$ is $\overline{\mathcal D}_A$-hypercyclic, then $\overline{\mathcal D}_A$-{\rm HC}$(T)$ is comeager in $X$?
\end{question} 

On the other hand, it might seem more than plausible that the two notions are not equivalent in the case of frequent hypercyclicity, \textit{i.e.} when $\mathcal F$ is the family of sets with positive lower density. But again, we don't know how to prove this.

\smallskip One may also think of ``local'' versions of hereditary frequent hypercyclicity. For example, one could say that an operator $T\in\mathfrak L(X)$ is
\begin{itemize}
\item[-] hereditarily $\mathcal F$-hypercyclic with respect to some sequence $(\Lambda_i)_{i\in\NN}\subset \mathcal F$ if, for any sequence $(A_i)\subset\mathcal F$ with $A_i\subset \Lambda_i$ and for any sequence of non-empty open sets $(V_i)$ in $X$, one can find a vector $x\in X$ such that $\mathcal N_T(x,V_i)\cap A_i\in\mathcal F$ for all $i\in\NN$;
\item[-] hereditarily $\mathcal F$-hypercyclic with respect to some set $\Lambda\in\mathcal F$ if it is hereditarily $\mathcal F$-HC with respect to the constant sequence $\Lambda_i=\Lambda$;
\item[-] hereditarily $\mathcal F$-hypercyclic in the weak sense with respect to some set $\Lambda\in\mathcal F$ is it is $\mathcal F_A$-hypercyclic for any $A\in\mathcal F\cap 2^\Lambda$.
\end{itemize}

\smallskip
When $\mathcal F$ is the family of all infinite subsets of $\NN$, hereditary $\mathcal F$-hypercyclicity in the weak sense with respect to some set $\Lambda=\{ n_k\,:\, k\geq 0\}$ is the same as hereditary hypercyclicity with respect to the sequence $(n_k)$ in the sense of \cite{BP-her}; and hence, by \cite[Theorem 2.3]{BP-her}, an operator $T$ is hereditarily $\mathcal F$-hypercyclic in the weak sense with respect to some set $\Lambda$ if and only it is topologically weakly mixing, \textit{i.e.} $T\oplus T$ is hypercyclic. Also, the proof of Proposition \ref{prop:sumhfhc} makes it clear that if $T$ is hereditarily $\mathcal F$-hypercyclic with respect to some sequence $(\Lambda_i)$, then $T\oplus T$ is $\mathcal F$-hypercyclic. This leads to
\begin{question} If $T\in\mathfrak L(X)$ is hereditarily $\mathcal F$-hypercyclic in the weak sense with respect to some set $\Lambda\in\mathcal F$, does it follow that $T\oplus T$ is $\mathcal F$-hypercyclic? And conversely?
\end{question}

In the same spirit and with \cite{EEM} in mind, one may ask
\begin{question} Does $\mathcal U$-frequent hypercyclicity imply some weak form of hereditary $\mathcal U$-frequent hypercyclicity, yet strong enough to ``explain'' why $T\oplus T$ is $\mathcal U$-frequently hypercyclic as soon as $T$ is?
\end{question}

\subsection{$\mathcal F$-transitivity and hereditary $\mathcal F$-transitivity} In topological dynamics, there is a natural notion of ``transitivity'' associated to a given Furstenberg family $\mathcal F$ {\rm (}see e.g. \cite{Gl}, and \cite{BeMePePu2} in the linear setting{\rm )}: if $X$ is a topological space, a continuous map $T:X\to X$ is said to be \textbf{$\mathcal F$-transitive} if $\mathcal N_T(U,V)\in \mathcal F$ for every pair $(U,V)$ of non-empty open sets in $X$, where \[\mathcal N_T(U,V):=\{ n\in\NN\,:\, T^n(U)\cap V\neq \emptyset\}.\]

\smallskip Following \cite{BeMePePu2}, one can consider a ``hereditary'' version of $\mathcal F$-transitivity: let us say that an operator $T\in\mathfrak L(X)$ is \textbf{hereditarily $\mathcal F$-transitive}  if  $\mathcal N(U, V)\cap A\in\mathcal{F}$ for every $A\in \mathcal{F}$ and all non-empty open sets $U,V$. There is an obvious link with hereditary $\mathcal F$-hypercyclicity.

\begin{remark}\label{trans} Hereditarily $\mathcal F$-hypercyclic operators are hereditarily $\mathcal F$-transitive.
\end{remark}
\begin{proof}
By Proposition \ref{densely}, we know that if $T$ is hereditarily $\mathcal{F}$-hypercyclic then $T$ is densely hereditarily $\mathcal{F}$-hypercyclic. Let $U, V$ be non-empty open sets in $X$, and let $A\in \mathcal{F}$. By dense hereditary $\mathcal F$-hypercyclicity, there exists $x\in U$ such that $N(x,V)\cap A\in \mathcal{F}$. In particular $N(U,V)\cap A\in \mathcal{F}$, so $T$ is  hereditarily $\mathcal{F}$-transitive.
\end{proof}

The converse is definitely not true in general, for the following reason: there exist topologically mixing operators that are not frequently hypercyclic. In particular, any such operator is  hereditarily $\underline{\mathcal{D}}$-transitive, where $\underline{\mathcal{D}}$ is the family of sets with positive lower density, but not frequently hypercyclic (\textit{i.e.} not $\underline{\mathcal{D}}$-hypercyclic). This leads to the following questions.

\begin{question} Are there operators which are frequently hypercyclic and topologically mixing, but not hereditarily frequently hypercyclic? 
\end{question}

\begin{question} Are there at least operators which are hereditarily $\underline{\mathcal{D}}$-transitive {and} frequently hypercyclic, but not hereditarily frequently hypercyclic? 
\end{question}

\smallskip
Given a Furstenberg family $\mathcal{F}$, one can define the \emph{dual family} $\mathcal{F}^*$ as the collection of all subsets $A$ of $\NN$ such that  $A\cap B\neq\emptyset$ for every $B\in\mathcal{F}$. It is clear by definition that every hereditarily $\mathcal{F}$-transitive operator is $\mathcal{F}^*$-transitive; and it is also clear that $(\underline{\mathcal{D}})^*=\overline{\mathcal{D}}_1$, the family of sets with upper density equal to $1$. Hence, every hereditarily frequently hypercyclic operator is $\overline{\mathcal{D}}_1$-transitive. It is natural to wonder if every frequently hypercyclic operator is $\overline{\mathcal{D}}_1$-transitive too. The next proposition shows that this is not the case. This is an improvement of \cite[Proposition~5.1]{BeMePePu2}, where it is shown that reiterative hypercyclicity does not imply $\overline{\mathcal{D}}_1$-transitivity. Moreover the example we give is any of the weighted shifts introduced in the proof of Theorem \ref{thm:sumfhc}; so this provides another proof that these shifts are not hereditarily frequently hypercyclic.

\begin{proposition}\label{propsansnom}
There exists a frequently hypercyclic weighted shift $B_w$ on $c_0(\ZZ_+)$ which is not $\overline{\mathcal{D}}_1$-transitive.
\end{proposition}
\begin{proof}
Let $B_w$ be one of the weighted shifts introduced in the proof of Theorem \ref{thm:sumfhc}.  By \cite[Proposition 3.3]{BeMePePu2}, in order to show that $B_w$ is not $\overline{\mathcal{D}}_1$-transitive, it is enough to find $M >0$ such that
\[{C}_M:= \{n\in\NN\, :\, |w_1\cdots w_n|>  M\}\not\in \overline{\mathcal{D}}_1.\]

With the notation of the proof of Theorem \ref{thm:sumfhc}, we know that for every $p\geq 1$, 
\[{C}_{2^p}\subset \bigcup_{q \ge p} (b_q\mathbb{N}+[-q,q])\cup \bigcup_{q\ge 1}\bigcup_{u\in {A}_{2q}}I_u^{a,4\varepsilon}.\]
Moreover, by assumption on $(b_q)$, we have
\[\lim_{p\to\infty} \overline{\text{dens}}\left(\bigcup_{q \ge p} (b_q\mathbb{N}+[-q,q])\right)=0;\]
and we also have
\[\bigcup_{q\ge 1}\bigcup_{u\in {A}_{2q}}I_u^{a,4\varepsilon}\subset \bigcup_{u\ge 1}I_u^{a,4\varepsilon}.\]
Since
\[\overline{\text{dens}}\Biggl(\bigcup_{u\ge 1}I_u^{a,4\varepsilon}\Biggr) \le \lim_{u\to\infty} \frac{\sum_{k=1}^u 8\varepsilon a^k}{(1+4\varepsilon)a^u}=\frac{8\varepsilon}{1+4\varepsilon}\sum_{k=0}^{\infty} a^{-k}<1\quad \text{if $a$ is sufficiently big,}\]
it follows that if $p$ is sufficiently big then
\[\overline{\text{dens}}\ {C}_{2^p}<1.\]

 This concludes the proof of Proposition \ref{propsansnom}.
\end{proof}

\smallskip

\subsection{About disjointness} The original definition of disjointness in topological dynamics goes back to Furstenberg's seminal paper \cite{Furdisj}.  The setting is that of compact dynamical systems $(X,T)$, \textit{i.e.} $X$ is a compact metric space and $T:X\to X$ is a continuous map. Two compact dynamical systems $(X_1, T_1)$ and $(X_2, T_2)$ are said to be disjoint if the only closed, $(T_1\times T_2)$-invariant set $\Gamma\subset X_1\times X_2$ such that $\pi_{X_1}(\Gamma)=X_1$ and $\pi_{X_2}(\Gamma)=X_2$ is $\Gamma= X_1\times X_2$. Note that since the spaces are compact, one could replace $\pi_{X_i}(\Gamma)$ by $\overline{\pi_{X_i}(\Gamma)}$ in the definition. For dynamical systems $(X,T)$ whose underlying space is not necessarily compact, both definitions make sense and lead to \textit{a priori} different notions of disjointness (the one ``with closure'' being stronger than the one ``without closure''). In particular, one could consider these notions in the linear setting. However, there are no disjoint pairs of linear dynamical systems in this sense, even ``without closures''. Indeed, if $T_1\in\mathfrak L(X_1)$ and $X_2\in\mathfrak L(X_2)$, then
$\Gamma:=( \{ 0\}\times X_2)\, \cup\,( X_1\times\{ 0\})$ shows that disjointness cannot be met. One can get round this difficulty by changing a little bit the definitions as follows: instead of  $\pi_{X_i}(\Gamma)=X_i$, require that $\pi_{X_i}\bigl( \Gamma\cap (X_1\setminus\{ 0\})\times (X_2\setminus\{ 0\})\bigr)=X_i\setminus\{ 0\}$; and likewise for the definition ``with closures''.

\smallskip Even though these definitions of disjointness are likely to be artificial, one can try to play a little bit with them. For example, copying out the relevant parts of \cite{Furdisj} -- namely, the proofs of Theorem II.1 and Theorem II.2 -- one gets the following results. Let us say that a linear dynamical system $(X,T)$ is \emph{minimal apart from $0$} if  every non-zero vector $x\in X$ is hypercyclic for $T$; equivalently, if the only  closed $T$-invariant subsets of $X$ are $\{ 0\}$ and $X$. Famous examples of Read \cite{Read} show that this can indeed happen.
\begin{proposition} Let $(X_1, T_1)$ and $(X_2, T_2)$ be two linear dynamical systems. If $(X_1, T_1)$ and $(X_2, T_2)$ are disjoint ``without closures'', then at least one of them is minimal apart from $0$.
\end{proposition}
\begin{proof} Assume that $(X_1, T_1)$ and $(X_2, T_2)$ are not minimal apart from $0$. Then, for $i=1,2$, one can find a closed $T_i$-invariant set $C_i\subset X_i$ such that $C_i\neq X_i$ and $C_i\cap (X_i\setminus\{ 0\})\neq\emptyset$; and $\Gamma:= (C_1\times X_2) \cup (X_1\times C_2)$ shows that $(X_1, T_1)$ and $(X_2, T_2)$ are not disjoint ``without closures''.
\end{proof}

\smallskip
\begin{proposition}\label{possible} Let $(X_1, T_1)$ and $(X_2, T_2)$ be two linear dynamical systems. Assume that the periodic points of $T_1$ are dense in $X_1$, and that $(X_2, T_2)$ is minimal apart from $0$. Then $(X_1, T_1)$ and $(X_2, T_2)$ are disjoint ``without closures''.
\end{proposition}
\begin{proof} Let $\Gamma \subset X_1\times X_2$ be a closed, $(T_1\times T_2)$-invariant set such that $\pi_{X_i}\bigl( \Gamma\cap (X_1\setminus\{ 0\})\times (X_2\setminus\{ 0\})\bigr)=X_i\setminus\{ 0\}$ for $i=1,2$. We have to show that $\Gamma=X_1\times X_2$; and since the periodic points of $T_1$ are dense in $X_1$, it is enough to show that $({\rm Per}(T_1)\setminus\{ 0\})\times X_2\subset\Gamma$.

Let $u\in X_1$ be any non-zero periodic point of $T_1$, and choose $d\in\NN$ such that $T^du=u$. By assumption on $\Gamma$, one can find $v\in X_2\setminus\{ 0\}$ such that $(u,v)\in \Gamma$. Then $(u, T_2^{dn}v)\in\Gamma$ for all $n\in\NN$ by $(T_1\times T_2)$-invariance of $\Gamma$. Moreover, $v\in \hc(T_2)$ by assumption on $(X_2,T_2)$. Hence, by Ansari's theorem \cite{Ansari}, $v$ is also a hypercyclic vector for $T_2^d$; and since $\Gamma$ is closed in $X_1\times X_2$, it follows that $\{ u\}\times X_2\subset\Gamma$.
 \end{proof}

\smallskip Proposition \ref{possible} implies in particular that linear dynamical systems which are disjoint ``without closure'' do exist. We don't know if this is also true ``with closure''; so one could think of possible weakenings of the definition of disjointness ``with closures''. For hypercyclic operators, one possible such weakening could be the following: one could say that two hypercyclic operators $T_1\in\mathfrak L(X_1)$ and $T_2\in\mathfrak L(X_2)$ are \emph{pseudo-disjoint} (just to give a name) if, whenever $x_1$ is hypercyclic vector for $T_1$ and $x_2$ is a hypercyclic vector for $T_2$, it follows that $(x_1, x_2)$ is hypercyclic for $T_1\times T_2$. This is indeed weaker than the definition of disjointness ``with closures'' (consider
$\Gamma:=\overline{{\rm Orb}\bigl((x_1, x_2), T_1\times T_2\bigr)}\,$), yet formally much stronger than the disjointness notion introduced in \cite{Bernal-disjoint} and \cite{BePe-disjoint}, \textit{i.e.} diagonal hypercyclicity. We are not much further ahead since we don't know if there are any pseudo-disjoint pairs of linear operators (whereas there are lots of interesting examples for diagonal hypercyclicity).   So we ask
\begin{question}\label{pseudo-disj} Are there any pseudo-disjoint pairs of operators, \textit{i.e.} pairs of hypercyclic operators $(T_1,T_2)$ such that $\hc(T_1)\times \hc(T_2)\subset \hc(T_1\times T_2)$? 
\end{question}

Regarding this question, one may observe that two linear operators $T_1$ and $T_2$ are trivially pseudo-disjoint if it happens that every vector $x\in X_1\times X_2$ with non-zero coordinates is hypercyclic for $T_1\times T_2$. This leads to the following ``strong'' form of Question \ref{pseudo-disj}.

\begin{question}\label{Readproduct} Are there pairs of operators $(T_1,T_2)$ such that  every $x\in (X_1\setminus\{ 0\})\times (X_2\setminus\{ 0\})$ is hypercyclic for $T_1\times T_2$?
\end{question}

Let us point  out the following amusing fact: if such a pair $(T_1,T_2)$ can be found, then the operator $T=T_1\times T_2$ acting on $X=X_1\times X_2$ is a hypercyclic operator such that $HC(T)$ is an open set but $HC(T)\neq X\setminus\{ 0\}$. We don't know of any example of operators with that property.

\smallskip
Finally, we note that if one extends the definition of pseudo-disjointness to possibly non-linear systems in the obvious way, it follows from the main result of \cite{ShkGp} that any irrational rotation of the circle is pseudo-disjoint from any hypercyclic operator. In view of that, one may consider the following variant of Question \ref{pseudo-disj}.
\begin{question} Are there natural classes of hypercyclic operators $\mathfrak C_1$, $\mathfrak C_2$ such that any $T_1\in\mathfrak C_1$ is pseudo-disjoint from any $T_2\in\mathfrak C_2$?
\end{question}

\subsection{Other questions} We conclude the paper by adding some other possibly interesting questions motivated by the results obtained in the paper.

\smallskip The first question asks for a converse to Observation \ref{prop:sumhfhc0}.

\begin{question} Let $\mathcal F$ be a Furstenberg family, and let $T\in\mathfrak L(X)$. Assume that $S\oplus T$ is $\mathcal F$-hypercyclic for every $\mathcal F$-hypercyclic operator $S$. Does it follow that $T$ is hereditarily $\mathcal F$-hypercyclic?
\end{question}

\smallskip The next two questions are related to the following ``trap'' into which it is easy to fall: if $\mathcal F$ and $\mathcal F'$ are two Furstenberg families, the fact that 
$\mathcal F\subset \mathcal F'$ does not formally imply that hereditary $\mathcal F$-hypercyclicity is a stronger property than hereditary $\mathcal F'$-hypercyclicity. 

\begin{question}\label{mix} Are there hereditarily frequently hypercyclic operators which are not topologically mixing?
\end{question}

\begin{question} Does hereditary frequent hypercyclicity imply hereditary $\mathcal U$-frequent hypercyclicity?
\end{question}

\smallskip  In the theory of frequently hypercyclic operators, there are non-trivial counterexamples to some tempting ``conjectures''. It is natural to ask if these examples are in fact hereditarily frequently hypercyclic, or if it is possible to modify them in order to get hereditarily frequently hypercyclic examples. In particular, with  \cite[Theorem 6.41]{BM09} and \cite{Me2} in mind, this leads to the following questions.

\begin{question} Are there hereditarily frequently hypercyclic operators which are not chaotic?
\end{question}

\begin{question} Are there invertible hereditarily frequently hypercyclic operators whose inverse is not frequently hypercyclic? 
\end{question}

One may also consider the following strengthened version of Question \ref{mix}, \textit{cf} \cite{BadGriv} or \cite[Theorem 6.45]{BM09}.
\begin{question} Are there operators which are both hereditarily frequently hypercyclic \emph{and} chaotic but not topologically mixing?
\end{question}

\smallskip The next two questions are related to the sufficient conditions we found for hereditary frequent hypercyclicity. Observe first that besides the Frequent Hypercyclicity Criterion and the unimodular eigenvectors machinery, there are other criteria to prove frequent hypercyclicity (see \cite{BeMePePu} or \cite[Theorem 5.35]{GMM21}). They do not imply hereditarily frequent hypercyclicity. Indeed, the criterion of \cite{BeMePePu}  is equivalent to frequent hypercyclicity for weighted shifts on $c_0$ whereas the $C$-type operator of Theorem \ref{Ctypeex} satisfies \cite[Theorem 5.35]{GMM21}.

\begin{question}
If $T\in\mathfrak L(X)$ is such that the $\mathbb T$-eigenvectors of $T$ are spanning with respect to  Lebesgue measure, does it follow that one can find a $T$-$\,$invariant measure $\mu$ with full support such that $(X,\mathcal B,\mu, T)$ is a factor of a dynamical system with countable Lebesgue spectrum? Does it follow at least that $T$ is hereditarily frequently hypercyclic?
\end{question}

\begin{question} Let $T\in\mathfrak L(X)$. Is it true that if the $\TT$-eigenvectors of $T$ are perfectly spanning, then $T$ is hereditarily frequently hypercyclic?%
\end{question}

\smallskip Concerning the invariant measure business, the next two questions seem natural. The first one is motivated by Theorem \ref{HUFHC}.

\begin{question} Does there exist an operator $T$ which is not hereditarily $\mathcal U$-frequenty hypercyclic but admits an ergodic measure with full support?
\end{question}

\begin{question}
If $X$ is a reflexive Banach space, then any frequently hypercyclic operator $T$ on $X$ admits a continuous invariant probability measure with full support {\rm (}see \cite{GriM}{\rm )}.
Is it possible to improve this result if $T$ is assumed to be hereditarily frequently hypercyclic?
\end{question}

\smallskip The next question is, of course, strongly reminiscent of the B\`es-Peris theorem \cite{BP-her}, according to which the Hypercyclicity Criterion characterizes topological weak mixing.
\begin{question} Given a Furstenberg family $\mathcal F$, is there some ``$\mathcal F$-hypercyclicity criterion'' characterizing the operators $T$ such that $T\oplus T$ is $\mathcal F$-hypercyclic?
\end{question}

\smallskip Finally, our last three questions concern the links between (hereditary) frequent hypercyclicity and the geometry of the underlying space $X$.
\begin{question} On which spaces $X$ is it possible to find hereditarily frequently hypercyclic operators? Is it possible at least on any complex Banach space admitting an unconditional Schauder decomposition?
\end{question}

\begin{question} Are there spaces $X$ which support frequently hypercyclic operators but no hereditarily frequently hypercyclic operator?
\end{question}

\begin{question} On which Banach spaces $X$ is it possible to construct $d$-frequently hypercyclic pairs $(T_1,T_2)$ which are not densely $d$-frequently hypercyclic?
\end{question}

\providecommand{\bysame}{\leavevmode\hbox to3em{\hrulefill}\thinspace}
\providecommand{\MR}{\relax\ifhmode\unskip\space\fi MR }
\providecommand{\MRhref}[2]{%
  \href{http://www.ams.org/mathscinet-getitem?mr=#1}{#2}
}
\providecommand{\href}[2]{#2}

\end{document}